\documentclass[a4paper]{article}
\usepackage[margin=1in]{geometry}
\usepackage[utf8]{inputenc}
\usepackage{amsmath}
\usepackage{extarrows}
\usepackage{amsthm}

\newtheorem{theorem}{Theorem}
\newtheorem{lemma}{Lemma}

\newtheorem{ass}{Assumption}

\newtheorem{rem}{Remark}
\newtheorem{inequality}{Inequality}
\newtheorem{corollary}{Corollary}

\newcommand{\RNum}[1]{\uppercase\expandafter{\romannumeral #1\relax}}
\usepackage{graphicx}
\usepackage{amssymb}
\usepackage[authoryear]{natbib}
\usepackage[colorlinks,citecolor=blue]{hyperref}
\usepackage[noend]{algpseudocode}
\usepackage{algorithmicx,algorithm}
\usepackage{xcolor}

\usepackage{bm}
\usepackage{appendix}
\usepackage{siunitx}
\usepackage{longtable,tabularx}
\usepackage{fontenc}
\usepackage{inputenc}
\usepackage{verbatim}
\usepackage{cases}
\usepackage{authblk}

\newcommand{\hatalpha}[1]{{\hat{\bm{\alpha}}}_{#1}}
\newcommand{\hatbeta}[1]{\hat{\beta}_{#1}}
\newcommand{\K}{L^{\frac{1}{2}}_{K}}

\newcommand{\hatfn}{\hat{f}_{n}}
\newcommand{\sumn}{\sum\limits_{i=1}^{n}}
\newcommand{\suminf}{\sum\limits_{k=1}^{+\infty}}
\newcommand{\innerproduct}[2]{\langle #1,#2 \rangle}
\newcommand{\opnorm}[1]{\|#1\|_{\operatorname{op}}}
\newcommand{\coeffone}{\omega^{-1}+\omega^{-\frac{1+\theta}{2}}\sqrt{c}}
\newcommand{\E}{\mathrm{E}}

\begin{document}
\title{Non-asymptotic Optimal Prediction Error for Growing-dimensional Partially Functional Linear Models}
\author[a,b]{Huiming Zhang}
\author[c]{Xiaoyu Lei\footnote{Correspondence author. Email: \texttt{leixy@uchicago.edu} (Xiaoyu Lei)}}
\affil[a]{Department of Mathematics, Faculty of Science and Technology, University of Macau, Macau, China.}
\affil[b]{Zhuhai UM Science \& Technology Research Institute, Zhuhai, China}
\affil[c]{Department of Statistics, University of Chicago, USA.}

\date{\today}
\maketitle

\begin{abstract}
Under the reproducing kernel Hilbert spaces (RKHS), we consider the penalized least-squares of the partially functional linear models (PFLM), whose predictor contains both functional and traditional multivariate parts, and the multivariate part allows a divergent number of parameters. From the non-asymptotic point of view, we focus on the rate-optimal upper and lower bounds of the prediction error. An exact upper bound for the excess prediction risk is shown in a non-asymptotic form under a more general assumption known as the effective dimension to the model, by which we also show the prediction consistency when the number of multivariate covariates $p$ slightly increases with the sample size $n$. Our new finding implies a trade-off between the number of non-functional predictors and the effective dimension of the kernel principal components to ensure prediction consistency in the increasing-dimensional setting. The analysis in our proof hinges on the spectral condition of the sandwich operator of the covariance operator and the reproducing kernel, and on sub-Gaussian and Berstein concentration inequalities for the random elements in Hilbert space. Finally, we derive the non-asymptotic minimax lower bound under the regularity assumption of the Kullback-Leibler divergence of the models.
\end{abstract}




\section{Introduction}\label{sec-1}
\hspace{4mm} Statistical analysis of functional data has become an important and challenging part in modern statistics since the leading work \cite{ramsay1982data} and pioneering paper \cite{grenander1950stochastic}. Due to technological innovation, the progress in data storage enables scientists to acquire complex data sets with the structures of curves, images, or other data with functional structures, referred to as functional data. Functional data analysis has a wide range of applications, including chemometrics, econometrics, and biomedical studies \cite{ramsay2007applied,kokoszka2017introduction}. There has been a large amount of works now focusing on many different non-parametric aspects of functional data such as kernel ridge regressions \cite{cai2006prediction,preda2007regression,du2014penalized,reimherr2018optimal}, penalized B-spline regressions \cite{cardot2003spline}, functional principal component regressions \cite{yao2005functional}, local linear regressions \cite{balllo2009local}, and reader can refer to the review paper \cite{wang2016functional} for more details.

\par Many existing works considering the estimation and prediction problems of functional data are based on the framework of functional principal component analysis (FPCA), see \cite{yao2005functional,cai2006prediction,hall2007methodology,zhou2022functional}. However, the predictive power of FPCA-based methods is weakened when the functional principal components cannot form an effective basis for the slope function, which often occurs in practice. A similar phenomenon also appears in principal component regressions, see \cite{jolliffe1982note}. An alternative method for the functional data is based on the reproducing kernel Hilbert spaces (RKHS) framework, which assumes the slope function is contained in an RKHS. It is shown in \cite{cai2012minimax} that the RKHS-based method performs better than the FPCA-based method when the slope function does not align well with the eigenfunctions of the covariance kernel. In fact, FPCA strongly relies on the leading principal scores with large eigenvalues correspondingly, and the eigenfunctions for representing the slope function inevitably lose some information for the response. From the view of machine learning theory, the FPCA is essentially a non-supervised method that often performs poorly in data analysis. For example, the analysis of Canadian weather data mentioned in \cite{cai2012minimax} and the Section 3 of \cite{cui2020partially}.

\par In this paper, we study the partially functional linear models (FPLM) containing both functional and multivariate parts in the predictor, which is originally considered in \cite{shin2009partial}. Let $\bm{X}=(X_{1},\cdots,X_{p})^{T}$ be the $p$-dimensional multivariate predictor, $Y(t)$ be the functional predictor, $\varepsilon$ be the random noise and $Z$ be the scalar response. In our work, we consider the PFLM taking the semi-parametric form
\begin{equation}\label{2-1}
Z=\bm{X}^{T}\bm{\alpha}_{0}+\int_{\mathcal{T}}Y(t)\beta_0(t)\mathrm{d}t+\varepsilon,
\end{equation}
where the $\beta_0(t)$ is the slope function for functional predictor and the $\bm{\alpha}_{0}$ is the regression coefficient for multivariate predictor. The model \eqref{2-1} contains both parametric and non-parametric part, which belongs to semi-parametric statistics. We assume the predictor to be the random design where $\bm{X}$ and $Y(t)$ are independent. Because the intercepts of predictor are easy to estimate by centralizing, for simplicity, we assume
\begin{center}
 $\operatorname{E}\bm{X}=\bm{0}$ and $\operatorname{E}Y(t)=0$.
\end{center}
Moreover, we assume the random noise $\varepsilon$ has conditional zero mean and finite variance provided the predictor $\bm{X}$ and $Y$.

For real data, suppose that we collect data $\{(Z_{i},\bm{X}_{i},Y_{i}(t),t \in \mathcal{T})\}_{i=1}^{n}$ that is i.i.d. (independent and identically distributed) drawn from $(Z,\bm{X},Y(t))$.

In some situations, many non-functional predictors are often collected for practical data analysis, and this increasing-dimensional setting has been considered in \cite{aneiros2015variable,kong2016partially}. Moreover, our work can also be applied to deal with divergent number of parameters. Theoretically, this setting requires assuming that the number of scalar covariates grows with the sample size, i.e., $p=p_n \to \infty$, and the convergence rate of the desired estimator becomes totally different from the case where the dimension of the non-functional predictors is fixed.

Dealing with the functional data as a stochastic process is a significant challenge in functional data analysis. Obviously, a functional covariate $Y(t)$  has an infinite number of predictors over the time domain (observed as discrete-time points) that are all highly correlated. The covariance function characterizes the correlation of the functional covariate. The estimation of the slope function in functional regressions is connected to ill-posed inverse problems. To handle the infinite-dimensionality of $\beta_0(t)$, people often impose certain regularity conditions on the hypothesized space of the slope function to ensure that the infinite-dimensional problems are tractable as a finite-dimensional approximation solution. Notwithstanding, the convergence rate of the slope estimators depends directly on the assumptions of the covariance operator's eigenvalue decay and the slope function's restricted space. Thus, the convergence rate cannot be parametric due to the infinite-dimensionality of the model \eqref{2-1}.

Some recent developments in PFLM include \cite{zhu2019estimation,cui2020partially}. Because of the shortcoming of the FPCA-based methods, we apply penalized least squares under the framework of RKHS here. There are few efforts on the non-asymptotic upper and lower bounds in the existing literature. For FPCA-based method, \cite{brunel2016non} considers the adaptive estimation procedure of functional linear models under a non-asymptotic framework; \cite{wahl2018a} analyses the prediction error of functional principal component regression
(FPCR) and proves a non-asymptotic upper bound for the corresponding squared risk. For the RKHS-based method, many works focus on the asymptotic results, such as \cite{cai2006prediction,cai2012minimax}. Under the framework of RKHS-based kernel ridge regressions, \cite{liu2020non} recently studies the non-asymptotic RKHS-norm error bounds (called oracle inequalities) for the estimated function $f_{0}$ in Gaussian non-parametric regression $Y=f_{0}\left(X\right)+\varepsilon, ~\varepsilon \sim N(0, \sigma^{2})$ where $f_0$ belongs to $L_2$. By applying the Matern kernel and supposing $f_{0}$ in a Holder space with the polynomial decay rate of eigenvalues $\lambda_{n}=O(n^{-2 a})$, \cite{liu2020non} derives the nearly minimax optimal convergence rate $\left(\frac{\log n}{n}\right)^{\frac{a}{2 a+1}}$ (up to a $\log n$ factor) for $L_2$-norm estimation error of the estimation of derivatives using plug-in kernel ridge regression (KRR) estimator.

To analyze the PFLM, the main innovation of our work is that we provide non-asymptotic upper and lower bounds for the excess prediction risk under the assumption of the effective dimension, which is equivalent to assuming the eigenvalues of the sandwich operator decay (see Remark \ref{rem-3}). \cite{tong2018analysis} establishes the upper bound for the excess prediction risk for the RKHS-based slope estimator of the functional linear models, but they do not consider the PFLM. Their result on the upper bound is a special case of our more general result because FLM is a special case of PFLM ($p=0$). If we let $p=0$ in the Theorem \ref{thm-1} and suppose $0 \cdot \ln{0} = 0$, we obtain the same non-asymptotic upper bound for FLM as the Theorem 3.6 in \cite{tong2018analysis} up to a constant. We also derive a minimax lower bound for the excess prediction risk under a general assumption concerning the Kullback-Leibler divergence of the model. In \cite{cai2012minimax}, the minimax lower bound is derived for the FLM in an asymptotic sense in Theorem 1, which is also a corollary of our result on the non-asymptotic minimax lower bound when $n \to \infty$. See the last paragraph in Section \ref{sec-4} for detailed derivation. Moreover, the optimal convergence rate (both upper and lower) of the excess prediction risk of PFLM is the same as that of FLM, which means the convergence of the functional part dominates the convergence of the PFLM.

The specific theoretical contributions of our work are listed as
\begin{itemize}
\item A significant contribution is that we obtain the non-asymptotic upper and lower bounds, which have not been well studied in the existing literature. We provide an exact non-asymptotic minimax lower bound on the excess prediction risk in PFLM. Moreover, a particular application of the proposed non-asymptotic version of the optimal prediction upper bound is that it allows analyzing the PFLM with a divergent number of non-functional predictors, which leads to the prediction consistency under the setting $p^{7}\log^{6}(p)=o(n)$.

\item We derive the non-asymptotic upper bound of the excess prediction risk for the RKHS-based least squares estimation in PFLM, and the optimal bound we obtain is more exact than that of \cite{cui2020partially} which only obtains the stochastic order of the convergence rate without the definite multiplying constants relevant to the high probability events. Our derivation for the optimal bound does not need the \emph{inverse Cauchy-Schwarz inequality} $$\operatorname{E}\left(\int Y(t) f(t) d t\right)^{4} \leq C\left[\operatorname{E}\left(\int Y(t) f(t) d t\right)^{2}\right]^{2},~\text{for}~f\in L^{2}(\mathcal{T})$$
as a moment assumption of the functional predictor. This condition is imposed in \cite{cui2020partially,cai2012minimax} to attain minimax prediction bounds for (partially) functional linear regressions. Our proof does not rely on the well-known representation lemma for the smoothing splines; see \cite{wahba1990spline,cucker2001on}.

\item The proof for the Theorem \ref{thm-1} is divided into three steps, and it relies on new non-trivial results. First, we prove the difference of the functional part between the true parameter and our least squares estimate is bounded. Second, based on the boundedness, we show the excess prediction risk contributed by the multivariate part of the predictor is convergent at $n^{-1}$-rate. Finally, according to the convergence of the multivariate part, we obtain the convergence of the prediction risk corresponding to the functional part in $n^{-\frac{1}{1+\theta}}$-rate, where $\theta$ is related to the effective dimension in the Assumption \ref{ass-4}. Specifically, the novelty of the proof lies in the Lemma \ref{3-2}, which is a crucial lemma for the Theorem \ref{thm-1}. In the Lemma \ref{3-2}, to show the concentration property of the random elements in Banach space, we use the methods in functional analysis and convert the random elements in Banach space to other relevant random elements in Hilbert space.

\end{itemize}

\par The outline of this paper is constructed as follows. In Section \ref{sec-2}, we provide the notations and definitions we need and a brief introduction to the RKHS and the PFLM. Section \ref{sec-3} shows our main theorem about the non-asymptotic upper bound for the excess prediction risk and two relevant corollaries. In Section \ref{sec-4}, we state the minimax lower bound for the excess prediction risk. In Section \ref{sec-5}, we provide the proof of the Theorem \ref{thm-1} in Section \ref{sec-3}. In Section \ref{sec-6}, we show the proof the Theorem \ref{thm-2Euclidean} in Section \ref{sec-4}. In Section \ref{sec-7} and \ref{sec-8}, we prove the lemmas we need for the proofs in Section \ref{sec-5} and \ref{sec-6}. In Section \ref{sec-9}, we summarize our conclusions and point out some future directions for research.

\section{Preliminaries}\label{sec-2}
\subsection{Notations and Definitions}
\hspace{4mm} Define $\|\bm{v}\|_2:=(\sum\limits_{i=1}^{p}v_{i}^{2})^{\frac{1}{2}}$ to be the $\ell_{2}$-norm of vector $\bm{v}\in\mathbb{R}^{p}$. Let $\mathcal{T}\subset\mathbb{R}$ be a compact set. Denote by $L^{2}(\mathcal{T})$ the Hilbert space composed by square integrable functions on $\mathcal{T}$, whose inner product and norm are respectively denoted by $\langle f,g \rangle$ and $\|f\|$ for any $f,g\in L^{2}(\mathcal{T})$.
\par Consider $T$ a bounded linear operator from a Banach space $A$ to a Banach space $B$ respectively endowed with the norms $\|\cdot\|_{A}$ and $\|\cdot\|_{B}$. Define the operator norm of $T$ as
\begin{equation*}
    \|T\|_{\mathrm{op}}:=\sup\limits_{x\in A:\|x\|_{A}=1}\|T(x)\|_{B}.
\end{equation*}
Let $T^{*}$ be the adjoint of $T$ from $B^{*}$ to $A^{*}$ defined by
$T^{*}(f)(x):=f(T(x)),\,\,\text{for any}\,\, f\in B^{*}.$ Notice the adjoint of an operator does not change the operator norm, and thus we have $\|T^{*}\|_{\mathrm{op}}=\|T\|_{\mathrm{op}}$.

\par For a matrix $E=(e_{i j})_{1\le i,j\le p}\in\mathbb{R}^{p\times p}$, when writing $\|E\|_{\mathrm{op}}$, we actually view $E$ as a bounded linear operator from $\mathbb{R}^{p}$ to $\mathbb{R}^{p}$ endowed with $\ell_{2}$-norm defined by $\bm{v}\mapsto E\bm{v}$, which is also called the spectral norm. Let $\|E\|_{\infty}:=\max\limits_{1\le i,j\le p}|e_{i j}|$ be the $\ell_{\infty}$-norm of the matrix $E$ and $\lambda_{\max}(E)$ be the maximal eigenvalue of the matrix $E$. Moreover, we have $\opnorm{E}\le p\|E\|_{\infty}$ from $5.6.\mathrm{P}23$ in Page 365 of \cite{horn2012matrix}.
\par For a real, symmetric, square integrable {on the domain $\mathcal{T}\times\mathcal{T}$}
{such that $ \int_{\mathcal{T}\times\mathcal{T}}R^2(s,t)\mathrm{d}s\mathrm{d}t<\infty$.}
The nonnegative definite function $R:\mathcal{T}\times\mathcal{T}\rightarrow\mathbb{R}$, {is called reproducing kernel in the following}. Let $L_{R}:L^{2}(\mathcal{T})\rightarrow L^{2}(\mathcal{T})$ be an integral operator (also a bounded linear operator) defined by
\begin{equation*}
    L_{R}(f)(t):=\langle R(s,t),f(s)\rangle =\int_{\mathcal{T}}R(s,t)f(s)\mathrm{d}s.
\end{equation*}
According to the {Hilbert-Schmidt theorem}, there exists a set of orthonormalized eigenfunctions $\{\psi^{R}_{k}:k\ge 1\}$ and a sequence of eigenvalues $\theta^{R}_{1}\ge \theta^{R}_{2}\ge \cdots >0$ such that
\begin{equation}\label{ser-rep}
    R(s,t)=\sum\limits_{k=1}^{+\infty}\theta^{R}_{k}\psi^{R}_{k}(s)\psi^{R}_{k}(t), \quad \forall s,t\in\mathcal{T}.
\end{equation}
see Theorem 4.6.5 in \cite{hsing2015theoretical} for the proof of such series representation \eqref{ser-rep}.

Noticing the orthonormality of the eigenfunctions $\{\psi^{R}_{k}\}_{k\ge 1}$, we have
\begin{equation*}
    L_{R}(\psi_{k}^{R})(s)=\innerproduct{R(s,t)}{\psi^{R}_{k}(t)}=\innerproduct{\sum\limits_{i=1}^{+\infty}\theta^{R}_{i}\psi^{R}_{i}(s)\psi^{R}_{i}(t)}{\psi^{R}_{k}(t)}=\sum\limits_{i=1}^{+\infty}\theta^{R}_{i}\psi^{R}_{i}(s)\innerproduct{\psi^{R}_{i}(t)}{\psi^{R}_{k}(t)}=\theta^{R}_{k}\psi^{R}_{k}(s).
\end{equation*}

In what follows, we let $\{(\theta^{R}_{k},\psi^{R}_{k})\}_{k\ge1}$ be the eigenvalue-eigenfunction pairs corresponding to the operator (or the equivalent bivariate function) $R$. Let $L_{R}^{\frac{1}{2}}$ be the operator satisfying $L_{R}^{\frac{1}{2}}(\psi^{R}_{k})=\sqrt{\theta^{R}_{k}}\psi^{R}_{k}.$ For two bivariate functions $R_{1},R_{2}:\mathcal{T}\times\mathcal{T}\rightarrow\mathbb{R}$, define
\begin{equation*}
    (R_{1}R_{2})(s,t):=\langle R_{1}(s,\cdot),R_{2}(\cdot,t) \rangle=
    \int_{\mathcal{T}}R_{1}(s,u)R_{2}(u,t)\mathrm{d}u.
\end{equation*}
Then we have the relation  and $L_{R_{1}R_{2}}=L_{R_{1}}\circ L_{R_{2}}$, where $\circ$ means the composition of mappings. To show $L_{R_{1}R_{2}}=L_{R_{1}}\circ L_{R_{2}}$, we notice
\begin{align*}
    L_{R_{1}}\circ L_{R_{2}}(f)(t)
    &=\int_{\mathcal{T}}R_{1}(t,s)L_{R_{2}}(f)(s)\mathrm{d}s =\int_{\mathcal{T}}R_{1}(t,s)\left(\int_{\mathcal{T}}R_{2}(s,u)f(u)\mathrm{d}u\right)\mathrm{d}s \\
    &=\int_{\mathcal{T}}\left(\int_{\mathcal{T}}R_{1}(t,s)R_{2}(s,u)\mathrm{d}s\right)f(u)\mathrm{d}u=L_{R_{1}R_{2}}(f)(t).
\end{align*}
\par Let $\mathrm{HS}(\mathcal{T})$ be the Hilbert space of the Hilbert-Schmidt operators on $L^{2}(\mathcal{T})$ with the inner product $\langle A,B\rangle_{H}:=\operatorname{Tr}(B^{*}A)$ and the norm $\|A\|_{\operatorname{HS}}^{2}=\sum\limits_{k=1}^{+\infty}\|A(\phi_{k})\|^{2}$ where $\{\phi_{k}\}_{k\ge1}$ is an orthonormal basis of $L^{2}(\mathcal{T})$. The space $\mathrm{HS}(\mathcal{T})$ is a subspace of the bounded linear operators on $L^{2}(\mathcal{T})$, with the norm relations $\|A\|_{\mathrm{op}}\le \|A\|_{\operatorname{HS}}$ and
$\|AB\|_{\operatorname{HS}}\le \|A\|_{\mathrm{op}}\|B\|_{\operatorname{HS}}$.

Given a reproducing kernel $K$, we can uniquely identify a RKHS $\mathcal{H}(K)$ composed by a subspace of $L^{2}(\mathcal{T})$ satisfying $K(t,\cdot)\in\mathcal{H}(K)$ for any $t\in\mathcal{T}$, which is endowed with an inner product $\langle\cdot ,\cdot \rangle_{K}$ such that
\begin{equation*}
    f(t)=\langle K(t,\cdot), f \rangle_{K} , \quad\text{for any}\quad f\in\mathcal{H}(K).
\end{equation*}
There is a well-known fact
$$L^{\frac{1}{2}}_K(L^{2}(\mathcal{T}))=\mathcal{H}(K),$$
i.e. the RKHS $\mathcal{H}(K)$ can be characterized as the range of $\K$ equipped with the norm $\|L_{K}^{1 / 2}(f)\|_{K}=\|f\|_{L^{2}(\mathcal{T})}$, see Corollary 1 in \cite{sun2005mercer} for details and extensions. For simplicity, let $\mathcal{H}(K)$ be dense in $L^{2}(\mathcal{T})$, which means $L^{\frac{1}{2}}_K$ is injective. The definition of $\K$ directly yields the compactness of $\K$, by noting that $\mathcal{T}$ is bounded and compact, the reproducing kernel $K$ is continuous, we have
\begin{center}
 $\kappa:=\opnorm{\K}<\infty$.
\end{center}
 Readers can refer to \cite{wahba1990spline,cucker2001on,hsing2015theoretical} for more discussions on RKHS.

For the i.i.d.  data $\{(Z_{i},\bm{X}_{i},Y_{i}(t),t \in \mathcal{T})\}_{i=1}^{n}$, define the empirical covariance matrix $D_{n}$ and the covariance matrix $D$ for the multivariate part of the predictor to be
\begin{equation*}
    D_{n}:=\frac{1}{n}\sum\limits_{i=1}^{n}\bm{X}_{i}\bm{X}_{i}^{T}
    \quad\text{and}\quad
    D:=\operatorname{E}(\bm{X}\bm{X}^{T}),
\end{equation*}
where we assume the expectations of random variables $\bm{X}\bm{X}^T$ exist. Let $\lambda_{\max}:=\lambda_{\max}(D)$ and $\lambda_{\min}:=\lambda_{\min}(D)$ be the maximal and minimal eigenvalue of the covariance matrix $D$. Similarly, when the expectation of $Y(s)Y(t)$ exists, define the empirical covariance function $C_{n}(s,t)$ and the covariance function $C(s,t)$ for the functional part of the predictor to be
\begin{equation*}
    C_{n}(s,t):=\frac{1}{n}\sum\limits_{i=1}^{n}Y_{i}(s)Y_{i}(t)
    \quad\text{and}\quad
    C(s,t):=\operatorname{E}(Y(s)Y(t)).
\end{equation*}
Given the asymmetric, square-integrable, and
non-negative definite covariance function $C(s,t)$, define the \textit{sandwich operator} of the covariance operator $C$ and the reproducing kernel $K$ by
\begin{center}
 $T=\K\circ L_{C}\circ\K$ and its empirical version $T_{n}=\K\circ L_{C_{n}}\circ\K$,
\end{center}
see \cite{cai2012minimax} for details. For simplicity of the following discussion, define
\begin{center}
$g_{n}:=\frac{1}{n}\sum_{i=1}^{n}\varepsilon_{i}\K Y_{i}$ and $a_{n}:=\frac{1}{n}\sum_{i=1}^{n}\varepsilon_{i}\bm{X}_{i}$,
\end{center}
 which are key quantities to derive the convergence rate of the desired estimator. Define the bounded linear operators $G_{n}:L^{2}(\mathcal{T})\rightarrow\mathbb{R}^{p}$ and $H_{n}:\mathbb{R}^{p}\rightarrow L^{2}(\mathcal{T})$ by
\begin{equation*}
G_{n}(f):=\frac{1}{n}\sum\limits_{i=1}^{n}\langle Y_{i},\K f \rangle\bm{X}_{i}, \quad\forall f\in L^{2}(\mathcal{T})
\quad\text{and}\quad
H_{n}(\bm{\alpha}):=\frac{1}{n}\sum\limits_{i=1}^{n}(\bm{X}_{i}^{T}\bm{\alpha}) \K Y_{i}, \quad \forall ~\bm{\alpha}\in \mathbb{R}^{p}.
\end{equation*}

\par Because of the compactness of operator $T$ and the Hilbert-Schmidt theorem on compact operator, see Theorem 11.3 in \cite{schechter2001principles}, there exist a set of eigenvalue-eigenfunction pairs $\{(\tau_{k},\varphi_{k}):k\ge 1\}$ such that the operation of $T$ can be decomposed in the following way
\begin{equation*}
    T(f) = \sum\limits_{k=1}^{\infty}\tau_k \langle f, \varphi_{k}\rangle\varphi_{k},
\end{equation*}
where $\{\varphi_k:k\ge 1\}$ is orthonormal basis and $\{\tau_k:k\ge 1\}$ decrease to $0$. Define the trace of the operator $(T+\lambda I)^{-1}T$ as
\begin{equation*}
    D(\lambda):=\operatorname{Tr}((T+\lambda I)^{-1}T),
\end{equation*}
{which is also called the effective dimension introduced to measure the convergence rate of the functional part; see \cite{zhang2005learning,caponnetto2007optimal}.}

\subsection{The Penalized Least Square for PFLM}
\hspace{4mm}
\par The goal of prediction given the predictor $\bm{X}$ and $Y(t)$ is to recover the prediction $\eta_{0}$, the right side of \eqref{2-1} without the random noise $\varepsilon$,
$\eta_{0}(\bm{X},Y(t)):=\bm{X}^{T}\bm{\alpha}_{0}+\int_{\mathcal{T}}Y(t)\beta_{0}(t)\mathrm{d}t.$ To estimate the true parameter $(\bm{\alpha}_{0},\beta_{0})$, the penalized least square is defined as
\begin{equation}\label{2-3}
    (\hat{\bm{\alpha}}_{n},\hat{\beta}_{n}):=\mathop{\arg\min}_{(\bm{\alpha},\beta)\in \mathbb{R}^{p}\times\mathcal{H}(K)}\frac{1}{n}\sum\limits_{i=1}^{n}\left( Z_{i}-\bm{X}^{T}_{i}\bm{\alpha}-\int_{\mathcal{T}}Y_{i}(t)\beta(t)\mathrm{d}t \right)^{2}+\lambda_{n}\|\beta\|^{2}_{K}.
\end{equation}
Noticing $L^{\frac{1}{2}}_{K}(L^{2}(\mathcal{T}))=\mathcal{H}(K)$, there exists $\hat{f_{n}}\in L^{2}(\mathcal{T})$ such that $L^{\frac{1}{2}}_{K}(\hat{f_{n}})=\hat{\beta_{n}}$. So the \eqref{2-3} is replaced by
\begin{equation}\label{2-4}
    (\hat{\bm{\alpha}}_{n},\hat{f}_{n}):=\mathop{\arg\min}_{(\bm{\alpha},f)\in \mathbb{R}^{p}\times L^{2}(\mathcal{T})}\frac{1}{n}\sum\limits_{i=1}^{n}\left( Z_{i}-\bm{X}^{T}_{i}\bm{\alpha}-\langle Y_{i},L^{\frac{1}{2}}_{K}f \rangle \right)^{2}+\lambda_{n}\|f\|^{2}.
\end{equation}
{For the Euclidean predictor vector, we assume the dimension of the multivariate parameter $p$ less than or equal to the number of the training samples $n$ ($p\le n$), by which the empirical covariance matrix $D_n$ is invertible almost surely according to the Theorem in \cite{okamoto1973distinctness}, if the distribution of $\{\bm{X}_{i}\}_{i=1}^{n}$ is absolutely continuous (with respect to Lebesgue measure).}
\begin{ass}\label{ass-0}
We assume $D_n$ and $D$ are positive definite for the given data with $p\le n$.
\end{ass}
{According to the definition of the penalized least squares $(\hatalpha{n},\hat{f}_{n})$, which minimize \eqref{2-4}, the difference between the penalized least squares $(\hatalpha{n},\hat{f}_{n})$ and the true parameter $(\bm{\alpha}_{0},f_{0})$ can be represented as
\begin{numcases}{}
\hatfn-f_{0}=-\lambda_{n}(T_{n}+\lambda_{n}I)^{-1}f_{0}-(T_{n}+\lambda_{n}I)^{-1}H_{n}(\hatalpha{n}-\bm{\alpha}_{0})+(T_{n}+\lambda_{n}I)^{-1}g_{n} ,  \label{2-5} \\
\hatalpha{n}-\bm{\alpha}_{0}=-D_{n}^{-1}G_{n}(\hatfn-f_{0})+D_{n}^{-1}a_{n} , \label{2-6}
\end{numcases}
where $T_{n}$, $H_{n}$, $G_{n}$, $D_{n}$, $g_{n}$ and $a_{n}$ are given in the previous section. The derivation of \eqref{2-5} and \eqref{2-6} is left to Section \ref{App-A}, where we use the method of the calculus of variations. Thus the existence of the minimizer of the penalized least squares $\eqref{2-4}$ becomes to find a tuple $(\hat{\bm{\alpha}}_n, \hat{f}_n)$ satisfying the equations above.}

Let $ \hat{\eta}_{n}(\bm{X},Y(t)):=\bm{X}^{T}\hatalpha{n}+\int_{\mathcal{T}}Y(t)\hatbeta{n}(t)\mathrm{d}t$ be the prediction rule induced by the penalized least square estimator $(\hatalpha{n},\hatbeta{n})$. For a prediction rule $\eta(\bm{X},Y(t))$, define the prediction risk to be
\begin{equation*}
    \mathcal{E}(\eta):=\operatorname{E}[Z^{*}-\eta(\bm{X}^{*},Y^{*}(t))]^{2},
\end{equation*}
where $(Z^{*},\bm{X}^{*},Y^{*}(t))$ is an independent copy of $(Z,\bm{X},Y(t))$. We measure the accuracy of the prediction $\hat{\eta}_{n}$ by the excess prediction risk
\begin{equation*}
    \mathcal{E}(\hat{\eta}_{n})-\mathcal{E}(\eta_{0})=\operatorname{E}[\hat{\eta}_{n}(\bm{X}^{*},Y^{*}(t))-\eta_{0}(\bm{X}^{*},Y^{*}(t))]^{2}.
\end{equation*}
Let $f_{0}\in L^{2}(\mathcal{T})$ satisfying $\K f_{0}=\beta_{0}$ and rewrite $\eta_{0}(\bm{X},Y(t))$ and $\hat{\eta}_{n}(\bm{X},Y(t))$ to
$$
\eta_{0}(\bm{X},Y(t))=\bm{X}^{T}\bm{\alpha}_{0}+\int_{\mathcal{T}}Y(t)(\K f_{0})(t)\mathrm{d}t~~\text{and}~~\hat{\eta}_{n}(\bm{X},Y(t))=\bm{X}^{T}\hatalpha{n}+\int_{\mathcal{T}}Y(t)(\K \hat{f}_{n})(t)\mathrm{d}t,
$$
by which we can bound the excess prediction risk
\begin{align*}
&~~~~\mathcal{E}(\hat{\eta}_{n})-\mathcal{E}(\eta_{0})
=\operatorname{E}\left[\bm{X}^{*T}(\bm{\alpha}_{0}-\hatalpha{n})+\int_{\mathcal{T}}Y^{*}(t)(\K (f_{0}-\hat{f}_{n}))(t)\mathrm{d}t\right]^{2}\\
&\le 2\operatorname{E}[\bm{X}^{*T}(\bm{\alpha}_{0}-\hatalpha{n})]^{2}+2\operatorname{E}\left[\int_{\mathcal{T}}Y^{*}(t)(\K (f_{0}-\hat{f}_{n}))(t)\mathrm{d}t\right]^{2}\\
&=2(\bm{\alpha}_{0}-\hatalpha{n})^{T}\operatorname{E}(\bm{X}^{*}\bm{X}^{*T})(\bm{\alpha}_{0}-\hatalpha{n})+2\iint_{\mathcal{T}\times\mathcal{T}}\operatorname{E}[Y^{*}(t)Y^{*}(s)](\K (f_{0}-\hat{f}_{n}))(t)(\K (f_{0}-\hat{f}_{n}))(s)\mathrm{d}t
\mathrm{d}s.\end{align*}

Notice the equality
\begin{equation}\label{2-7}
\operatorname{E}\left[\int_{\mathcal{T}}Y(t)f(t)\mathrm{d}t\right]^{2}
=\iint_{\mathcal{T}\times\mathcal{T}}\operatorname{E}[Y(s)Y(t)]f(s)f(t)\mathrm{d}s\mathrm{d}t =\int_{\mathcal{T}}f(t)\left( \int_{\mathcal{T}}C(s,t)f(s)\mathrm{d}s \right)\mathrm{d}t
=\langle f,L_{C}f   \rangle.
\end{equation}

With the definitions above, we reformulate the upper bound for the excess prediction risk to
\begin{align}\label{excess}
\mathcal{E}(\hat{\eta}_{n})-\mathcal{E}(\eta_{0})
&\le 2\lambda_{\max}\|\hatalpha{n}-\bm{\alpha}_{0}\|_2^{2}+
2\langle \K(\hat{f}_{n}-f_{0}),L_{C}\K(\hat{f}_{n}-f_{0}) \rangle  \nonumber\\
&= 2\lambda_{\max}\|\hatalpha{n}-\bm{\alpha}_{0}\|_2^{2}+2\|T^{\frac{1}{2}}(\hat{f}_{n}-f_{0})\|^{2},
\end{align}
which is relatively easy to analyze.

\section{The Analysis and Main Results}\label{sec-3}
\subsection{The Analysis}
\hspace{4mm} Based on the RKHS framework, more regularity assumptions are needed to ensure our main results.
\begin{ass}\label{ass-1}
Let ${\left\| X \right\|_{G}} :=\sup_{k \ge 1} {[ {\frac{{\rm{E}}{X^{2k}}}{(2 k-1) ! !}} ]^{1/(2k)}}$ be the sub-Gaussian norm. We assume $\{X_{j}\}_{j=1}^p$ satisfies the sub-Gaussian growth of moments condition, i.e.
\begin{equation}\label{ass-1-2}
  \max_{1\le j\le p}{\left\|X_{j} \right\|_{G}}\le M_1<\infty
\end{equation}
and $\max_{1\le j\le p}{\E |X_{j}|^2}\le v^2<\infty$.
\end{ass}
\begin{ass}\label{ass-2}
$Y(t)$ is a bounded square integrable stochastic process: there exists $M_{2}>0$ such that $\| Y(\cdot) \|_{L^2(\mathcal{T})} \le M_2$ (a.s.).
\end{ass}
\begin{ass}\label{ass-3}
{The random noise $\varepsilon$ has conditional zero mean and finite variance: $\operatorname{E}(\varepsilon|\bm{X},Y)=0$ and $\operatorname{E}(\varepsilon^{2}|\bm{X},Y)<\sigma^{2}$ given $\bm{X}$ and $Y$.}
\end{ass}
\begin{ass}\label{ass-4}
The effective dimension of $T$ satisfies $D(\lambda):=\operatorname{Tr}((T+\lambda I)^{-1}T)\le c\lambda^{-\theta}$ for constants $c>0$ {and $0<\theta \le 1$}.
\end{ass}

{The sub-Gaussian data Assumption \ref{ass-1} has been adopted in many high-dimensional statistics references, see \cite{zhang2020concentration} for a review. } We assume that $Y(\cdot)$ is bounded in $L^2$ norm almost surely in the Assumption \ref{ass-2}. {The Assumption \ref{ass-3} follows the general assumptions of conditional zero mean and finite variance on random noise given the observations $\bm{X}$ and $Y(t)$.} The Assumption \ref{ass-4} on the effective dimension has been adopted in \cite{tong2018analysis}, which reflects the convergence of eigenvalues of $L_{C}$ and $L_{K}$ and how their eigenfunctions align. {The Assumption \ref{ass-4} is also equal to the common assumption on the decay rate of the eigenvalues of $T$ (see Remark \ref{rem-3}).}

\begin{rem}
For the simplicity of later proof and statement of lemmas, we further assume $\kappa$, $M_{2}>1$. Because of the boundedness, these assumptions make no essential difference compared with the original assumptions.

Define a random variable $\xi$ taking values in $\mathrm{HS}(\mathcal{T})$ by
\begin{equation*}
    \xi(f):=(T+\lambda_{n}I)^{-\frac{1}{2}}\langle \K Y,f \rangle\K Y.
\end{equation*}
Actually, we can weaken the boundedness in Assumption \ref{ass-2} by assuming the Bernstein's growth of moments condition of $\xi$ in $\mathrm{\mathrm{HS}}(\mathcal{T})$: there exist $\tilde M, \tilde\nu>0$ such that
\begin{equation}\label{cond-1}
    \operatorname{E}(\|\xi-\operatorname{E}\xi\|^{l}_{\operatorname{HS}})\le \frac{D(\lambda_{n})\tilde M^{2}}{2}\tilde\nu^{l-2}l!\, ,\quad \text{for all integer}\,\, l\ge 2.
\end{equation}
Then using the Lemma \ref{3-8}, we obtain a similar result as in the Lemma \ref{3-1}, by which we have an analogous result for the non-asymptotic optimal prediction error as in the Theorem \ref{thm-1}. But the condition \eqref{cond-1} is challenging to verify, and a similar situation is also provided for the FPCA method in H2 of \cite{brunel2016non}. Here, we do not offer the complete proof under the condition \eqref{cond-1}.
\end{rem}

\par Before getting to our main results, we need two important lemmas, of which the proofs are left to Section \ref{sec-7} and \ref{sec-8}. The following Lemma \ref{3-1} and Lemma \ref{3-2} can be viewed as the concentration inequalities for the operator-valued random variables $T_{n}$, $G_{n}$ and $H_{n}$. The concentration inequalities for the random variable taking values in Hilbert space, as stated in the Lemma \ref{3-3} and \ref{3-8} play an important role in the proofs of lemmas.
\begin{lemma}\label{3-1}
Under the Assumption \ref{ass-2}, for any $\delta_{1}\in(0,2e^{-1})$, with probability at least $1-\delta_{1}$, we have
\begin{center}
$\|(T+\lambda_{n}I)^{-\frac{1}{2}}(T_{n}-T)\|_{\operatorname{op}}
    \le c_{1}\log\left(\frac{2}{\delta_{1}}\right)B_{n},$ where
$c_{1}:=2\kappa^{2}M^{2}_{2}$ and $B_{n}:=\frac{1}{n\sqrt{\lambda_{n}}}+\sqrt{\frac{D(\lambda_{n})}{n}}$.
\end{center}
\end{lemma}

\begin{lemma}[$\|G_{n}\|_{\operatorname{op}}=\|H_{n}\|_{\operatorname{op}}$]\label{3-2}
 Under the Assumptions \ref{ass-1} and \ref{ass-2}, for any $\delta_{2}\in(0,1)$, with probability at least $1-\delta_{2}$, we have
\begin{center}
$\|G_{n}\|_{\operatorname{op}}=\|H_{n}\|_{\operatorname{op}}\le \frac{\kappa M_2[v+8M_1{\log^{1/2} ({p}/{\delta_{2}})}]}{\sqrt n}$.
\end{center}
\end{lemma}

\subsection{Main Results}
\hspace{4mm}With all the preparations above, we can state this paper's main result. The following theorem provides a non-asymptotic upper bound for the excess prediction risk.
\begin{theorem}\label{thm-1}
Under the Assumptions \ref{ass-0}-\ref{ass-4}, for any $\delta_1,\delta_3,\delta_4,\delta_5\in(0,1)$, $\delta_{2}\in(0,2e^{-1})$, let $\lambda_{n}:=\omega n^{-\frac{1}{1+\theta}},(\omega >0)$,
{$N_{1}:=24p\opnorm{D^{-1}}(48p\opnorm{D^{-1}}M_1^4+M_1^2)\log\left(\frac{2p^{2}}{\delta_{5}}\right)$} and
\begin{center}
$N_{2}=\left[ \frac{12p^{2}\kappa^{2}(\upsilon+M_{1})^{2}M_{2}^{2}\opnorm{D^{-1}}}{\omega}\log^{3}\left(\frac{2p}{\delta_{2}}\right)\log\left(\frac{2p}{\delta_{5}}\right) \right]^{\frac{1+\theta}{\theta}}, $
\end{center}
such that for $n\ge n_{0}:=\lceil \max\{N_{1},N_{2}\} \rceil$, we have with probability at least $1-\sum_{i=1}^{5}\delta_{i}$
\begin{equation}\label{pre}
    \mathcal{E}(\hat{\eta}_{n})-\mathcal{E}(\eta_{0})\le
    \left(2\lambda_{\max}(2c_{4}c_{6}+c_{5})^{2}+2c_{9}^{2}\right)n^{-1}+\left(4(c_{7}+c_{8})c_{9}\sqrt{\omega}\right)n^{-\frac{2+\theta}{2+2\theta}}+\left(2(c_{7}+c_{8})^{2}\omega\right)n^{-\frac{1}{1+\theta}},
\end{equation}
where $\{c_{i}\}_{i=4}^9$ are specific constants given in the proof that depend on the true parameters and the assumptions, and can be written as {$c_{4}:={\kappa M_2[v+8M_1{\log^{1/2} ({p}/{\delta_{2}})}]}\frac{3\opnorm{D^{-1}}}{2}$, $c_{5}=\frac{3\sqrt{p}\sigma v\opnorm{D^{-1}}}{2\sqrt{\delta_{4}}},$}
\begin{center}
 $c_{6}:=\|f_{0}\|+\frac{2p\kappa(\upsilon+M_{1})M_{2}c_{5}}{\omega}\log\left(\frac{2p}{\delta_{2}}\right)+\frac{\sigma(\omega^{-1}+\omega^{-\frac{1+\theta}{2}}\sqrt{c})}{\sqrt{\delta_{3}}}, $~$c_{7}:=\|f_{0}\|\left[2\kappa^{2}M_{2}^{2}\left(\omega^{-1}+\omega^{-\frac{1+\theta}{2}}\sqrt{c}\right)\log\left(\frac{2}{\delta_{1}}\right)+1\right], $
\end{center}
\begin{center}
$c_{8}:=\left( 2\kappa^{2}M_{2}^{2}\left(\omega^{-1}+\omega^{-\frac{1+\theta}{2}}\sqrt{c}\right)\log\left(\frac{2}{\delta_{1}}\right)  +1\right)^{2}\frac{\sigma(\omega^{-1}+\omega^{-\frac{1+\theta}{2}}\sqrt{c})}{\sqrt{\delta_{3}}}$ and
\end{center}
\begin{center}
{$  c_{9}:=\frac{\kappa M_2(2c_{4}c_{6}+c_{5})}{\sqrt{\omega}}\left( c_{1}(\coeffone)\log\left(\frac{2}{\delta_{1}}\right)  +1\right)\left[v+8M_1{\log^{1/2} \left(\frac{p}{\delta_{2}}\right)}\right].$}
\end{center}
\end{theorem}

Equation \eqref{pre} presents an exact upper bound of the excess perdition risk with all precise constants determined by the regularity conditions. The first term on the right side of \eqref{pre} is ascribed to the parametric part of the PFLM. The second term is a mixed bound consisting of both the parametric and the functional part since the prediction risk is a square function composed of both the functional and non-functional predictors. The last term is a dominated term, which reveals that the signal strength $||f_0||$, the operator norm of the reproducing kernel, and the variation of functional predictor play a crucial role in the non-asymptotic upper bound of prediction risk. Unfortunately, these assumption-dependent constants are always ignored in most references of asymptotic analysis for functional regressions.

\begin{rem}\label{rem-4}
{At first glance, it seems very surprising that the upper bound on the excess risk is independent of $\alpha_0$. From \eqref{2-6},
\begin{equation*}
 \bm{\hat\alpha}_{n}-\bm{\alpha}_{0}=\left(\frac{1}{n}\sum\limits_{i=1}^{n}\bm{X}_{i}\bm{X}_{i}^{T}\right)^{-1}\cdot\left(\frac{1}{n}\sum\limits_{i=1}^{n}\langle Y_{i},\K ( \hat{f}_n - f_{0} ) \rangle\bm{X}_{i}+ \frac{1}{n}\sumn\varepsilon_{i}\bm{X}_{i}\right),
\end{equation*}
and thus $\bm{\hat\alpha}_{n}-\bm{\alpha}_{0}$ has no relation to $\alpha_0$. We illustrate this phenomenon by setting $\beta_0(t)=0$ in our model (1). Then the PFLM degenerates to the classical linear models with a divergence number of Euclidean predictors. We have $\hat{\bm{ \alpha}}_{n}-\bm{\alpha}_{0}=(\frac{1}{n}\sum_{i=1}^{n}\bm{X}_{i}\bm{X}_{i}^{T})^{-1}\cdot(\frac{1}{n}\sum_{i=1}^{n}\varepsilon_{i}\bm{X}_{i}),$ which is free of $\bm{\alpha}_{0}$.}
\end{rem}

From the proof of the Theorem \ref{thm-1}, one can obtain the non-asymptotic upper bounds of functional and the non-functional parameters regressions below.
\begin{corollary}\label{confidence}
Under the conditions in Theorem \ref{thm-1}, for $n > n_{0}$ we have,
\begin{equation}\label{ALFA0}
   P\left( \|\hatalpha{n}-\bm{\alpha}_{0}\|_2\le
    \frac{2c_{4}c_{6}+c_{5}}{\sqrt{n}}\right)\ge 1-\delta_{2}-\delta_{3}-\delta_{4}-\delta_{5},
\end{equation}
\begin{equation}\label{beta0}
  P\left(  \|T^{\frac{1}{2}}(\hatfn-f_{0})\|\le (c_{7}+c_{8})\sqrt{\lambda_{n}}+\frac{c_{9}}{\sqrt{n}}\right)\ge 1-\sum\limits_{i=1}^{5}\delta_{i}.
\end{equation}
\end{corollary}

It should be noted that we are unable to show the asymptotic normality of the non-functional parameters $\hatalpha{n}-\bm{\alpha}_{0}$ because it is influenced by the functional parameter $\hatfn-f_{0}$ as shown in \eqref{2-6}. And it is difficult to derive an analogy of the central limit theorem for the functional parameter.

The \eqref{ALFA0} and \eqref{beta0} in Corollary \ref{confidence} are useful high-probability events, which can be used to obtain the confidence balls for $\bm{\alpha}_{0}$ and $f_{0}$ under the distance $\|\hatalpha{n}-\bm{\alpha}_{0}\|^{2}$ and $\|T^{\frac{1}{2}}(\hat{f}_{n}-f_{0})\|^{2}$.  They are also helpful for constructing testing statistics, and thus they conceive non-asymptotic hypothesis testing for functional regressions, see \cite{yang2020non} for the case of non-parametric regressions.

\par Another corollary of the Theorem \ref{thm-1} is the excess prediction risk $\mathcal{E}(\hat{\eta}_{n})-\mathcal{E}(\eta_{0})=O_{p}(n^{-\frac{1}{1+\theta}})$. From the proof, we notice the convergence rate of the prediction risk contributed by the multivariate part of the predictor is $O_{p}(n^{-1})$, faster than the convergence rate corresponding to the functional part of the predictor, which is $O_{p}(n^{-\frac{1}{1+\theta}})$. Therefore, the convergence rate of the prediction risk of the partially functional linear model is the same as the optimal rate for the functional linear model \cite{cai2012minimax}.

\begin{rem}\label{rem-3}
{The assumption that the eigenvalues $\{\tau_{k}\}_{k \ge 1}$ decay as $\tau_{k}\le c^{\prime}k^{-2r}\,(r>\frac{1}{2})$ is equivalent to our assumption \ref{ass-4}. For one direction, see the following derivation.}
\begin{align*}
    D(\lambda_{n})
    &=\suminf\frac{\tau_{k}}{\tau_{k}+\lambda_{n}}\le \suminf\frac{c^{\prime}k^{-2r}}{c^{\prime}k^{-2r}+\lambda_{n}}=\suminf\frac{c^{\prime}}{c^{\prime}+\lambda_{n}k^{2r}}\le \int_{0}^{+\infty}\frac{c^{\prime}}{c^{\prime}+\lambda_{n}t^{2r}}\mathrm{d}t \\
    &\xlongequal{s=\lambda_{n}^{\frac{1}{2r}}t} \lambda_{n}^{-\frac{1}{2r}}\int_{0}^{+\infty}\frac{c^{\prime}}{c^{\prime}+s^{2r}}\mathrm{d}s \lesssim \lambda_{n}^{-\frac{1}{2r}}\asymp n^{\frac{1}{1+2r}}~~~(\lambda_{n}=\omega n^{-\frac{2r}{1+2r}}).
\end{align*}
\end{rem}
\begin{corollary}
Suppose the Assumptions \ref{ass-0}-\ref{ass-3} are satisfied. Assume the eigenvalues $\tau_{k}$ decay as $\tau_{k}\le c^{\prime}k^{-2r}$ for some $c^{\prime}>0$ and $r>\frac{1}{2}$. For any $\delta_1,\delta_3,\delta_4,\delta_5\in(0,1)$ and $\delta_{2}\in(0,2e^{-1})$, by taking $\lambda_{n}=\omega n^{-\frac{2r}{1+2r}}$, there exists an integer $n_{0}$ such that for $n>n_{0}$, we have with probability at least $1-\sum_{i=1}^{5}\delta_{i}$
\begin{equation*}
    \mathcal{E}(\hat{\eta}_{n})-\mathcal{E}(\eta_{0})\le
    (2\lambda_{\max}(2c_{4}c_{6}+c_{5})^{2}+2c_{9}^{2})n^{-1}+(4(c_{7}+c_{8})c_{9}\sqrt{\omega})n^{-\frac{1+4r}{2+4r}}+(2(c_{7}+c_{8})^{2}\omega)n^{-\frac{2r}{1+2r}},
\end{equation*}
where $c_{i}\,(4\le i\le 9)$ and $n_{0}$ are the same as those of the Theorem \ref{thm-1} except replacing $\theta$ by $\frac{1}{2r}$ and $c$ by a constant relevant to $c^{\prime}$ and $r$.
\end{corollary}

A valuable and insightful application of the Theorem \ref{thm-1} is that we can consider the situation where the number of multivariate covariates $p$ increases as a function of $n$. {We need the following assumptions in the increasing dimension background.
\begin{ass}\label{ass-5}
The number of the multivariate covariates $p=p_{n}$ can increase as a function of sample size $n$. We write $n \gg f(p)$ when $f(p)n^{-1}\rightarrow 0$.
\end{ass}
\begin{ass}\label{ass-8}
The minimal and maximal eigenvalue of $D$ are bounded from below and above as $n$ increases: there exist positive constants $m^{\prime}$ and $M^{\prime}$ such that $m^{\prime} < \lambda_{\min}(D)<\lambda_{\max}(D) < M^{\prime}$ for all $n$ and $p$.
\end{ass}}

{According to the definition of $c_{i}\,(4\le i\le 9)$ and $N_{i}\,(i=1,2)$ in the Theorem \ref{thm-1} and the Assumption \ref{ass-8} on the eigenvalues of $D$, we have the following estimation on the asymptotic order of each coefficients in the Theorem \ref{thm-1}.}
\begin{align*}
    &c_{4}=O(p\log\left(p\right)),\quad c_{5}=O(p^{\frac{1}{2}}),\quad c_{6}=O(p^{\frac{3}{2}}\log(p)),\quad c_{7}=c_{8}=O(1),\\
    & c_{9}=O(p^{\frac{7}{2}}\log^{3}(p)),
    \quad N_{1}=O(p^{2}\log(p))\quad\text{and}\quad N_{2}=O(p^{\frac{2(1+\theta)}{\theta}}\log^{\frac{4(1+\theta)}{\theta}}(p)).
\end{align*}
From these orders, it implies
\begin{align*}
    &(2\lambda_{\max}(2c_{4}c_{6}+c_{5})^{2}+2c_{9}^{2})n^{-1}=O(p^{7}\log^{6}(p)n^{-1}),\quad
    (4(c_{7}+c_{8})c_{9}\sqrt{\omega})n^{-\frac{2+\theta}{2+2\theta}}=O(p^{\frac{7}{2}}\log^{3}(p)n^{-\frac{2+\theta}{2+2\theta}})\\
    &\text{and}\quad (2(c_{7}+c_{8})^{2}\omega)n^{-\frac{1}{1+\theta}}=O(n^{-\frac{1}{1+\theta}}).
\end{align*}
{
When letting $p^{7}\log^{6}(p)n^{-1}\rightarrow 0$, i.e. $n\gg O(p^{7}\log^{6}(p))$, we have as $n,p\rightarrow\infty$,
\begin{equation*}
    p^{\frac{7}{2}}\log^{3}(p)n^{-\frac{2+\theta}{2+2\theta}}\ll O( n^{\frac{1}{2}-\frac{2+\theta}{2+2\theta}})\rightarrow 0,
\end{equation*}
from which we see the upper bound in the Theorem \ref{thm-1} converges to $0$.}

{To apply the Theorem \ref{thm-1} in the increasing dimension background, except the convergence of the upper bound, we also need the condition $n > N_1$ and $n > N_2$ satisfied as $n,p\rightarrow \infty$. Notice $n\gg O(p^{7}\log^{6}(p))\gg O(p^{2}\log(p))=N_{1}$. If we let  $\frac{2(1+\theta)}{\theta}<7\Leftrightarrow \theta>\frac{2}{5}$, we have $n\gg N_{2}$ under the condition $n\gg O(p^{7}\log^{6}(p))$ after noticing $p\gg\log^{\varepsilon}(p)$ for any $\varepsilon\in\mathbb{R}$ and the asymptotic order of $N_{2}=O(p^{\frac{2(1+\theta)}{\theta}}\log^{\frac{4(1+\theta)}{\theta}}(p))$. Therefore we have the following prediction consistency for the increasing dimension situation of non-functional parameters.}

\begin{corollary}\label{col-2}
Under the Assumptions \ref{ass-0}-\ref{ass-8}, if the constant $\frac{2}{5}<\theta\le 1$ in the Assumption \ref{ass-4} and  $p^{7}\log^{6}(p)=o(n)$ in the Assumption \ref{ass-5}, we have the consistency for the excess prediction risk: $$\mathcal{E}(\hat{\eta}_{n})-\mathcal{E}(\eta_{0})=o_{p}(1).$$
\end{corollary}
\begin{rem}
If we assume the eigenvalues $\tau_{k}$ decay as $\tau_{k}\le c^{\prime}k^{-2r}\,(r>\frac{1}{2})$ in the increasing-dimensional setting, by applying the Corollary \ref{col-2} and noticing $\theta=\frac{1}{2r}$, we need to further assume $r<\frac{5}{4}$ to obtain the prediction consistency, which means the convergence rate of eigenvalues can not be too fast. Intuitively, when $p$ increases, $r$ can not be too large or equivalently, the effective dimension $D(\lambda_{n})\asymp n^{\frac{1}{1+2r}}$ can not be too small. It implies we need to find a trade-off between the number of non-functional predictors and the effective dimension to get the prediction consistency.
\end{rem}

The prediction consistency theory has been well-established for non-parametric and high-dimensional statistics; see \cite{zhuang2018maximum} for the recent development of general regularized maximum likelihood estimators. However, their works mainly aim for non-parametric or high-dimensional models and do not cover the semi-parametric case as studied in our paper.

\section{Minimax Lower Bound}\label{sec-4}
\hspace{4mm} {In this section, we derive a minimax lower bound for the excess prediction risk in the following Theorem \ref{thm-2Euclidean} when $p \to \infty$. If $p$ is fixed, the minimax lower bound is postponed at the end of Section \ref{se:FIX}.}  To verify the optimality of the upper bound of the prediction risk for the proposed estimator, the result on the minimax lower bound below shows the prediction risk of our estimator achieves the theoretical lower bound caused by the intrinsic limitation of the PFLM. Let $P_{\bm{\alpha}_{0},\beta_{0}}$ be the probability taken over the space $(Z,\bm{X},Y)$ where $Z$ is generated by the true parameter $Z=\bm{X}^{T}\bm{\alpha}_{0}+\innerproduct{Y}{\beta_{0}}+\varepsilon$. Before stating the main result, we need a regularity assumption relevant to the Kullback-Leibler of the random noises.

\begin{ass}\label{ass-6}
{ For different $\beta_{1},\beta_{2}\in\mathcal{H}(K)$ and $\bm{\alpha}_1, \bm{\alpha}_2 \in\mathbb{R}^{p}$. We assume that the Kullback-Leibler distance between $P_{\alpha_1,\beta_1}$ and $P_{\bm\alpha_2,\bm\beta_{2}}$ can be bounded by
$$K(P_{\bm\alpha_1,\beta_1} | P_{\bm\alpha_2,\beta_{2}})=\operatorname{E}_{\bm\alpha_1,\beta_1}\left(\log\left( \frac{\mathrm{d}P_{\bm\alpha_1,\beta_1}}{\mathrm{d}P_{\alpha_2,\beta_2}}\right)\right)\le K_{\sigma^{2}} \operatorname{E}[\innerproduct{Y}{\beta_{1}-\beta_{2}}+\bm{X}^{T}(\bm\alpha_1-\bm\alpha_2)]^{2},
$$
where $K_{\sigma^{2}}>0$ is a variance-dependent constant and $\operatorname{E}_{\bm\alpha_1,\beta_1}$ means the expectation is taken over $P_{\bm\alpha_1,\beta_1}$.}
\end{ass}

The examples of constant $K_{\sigma^{2}}$ include  noises of exponential families (see \cite{abramovich2016model,du2014penalized}) and noises with self-concordant log-density function (see \cite{ostrovskii2018finite}). If we assume the random noise $\varepsilon\sim N(0,\sigma^{2})$, the constant
\begin{equation}\label{Ksigma2}
K_{\sigma^{2}}=\frac{1}{2\sigma^{2}}.
\end{equation}
The proof is left to Section \ref{App-C}. Now we state the main theorem of this section, of which the proof is left to Section \ref{sec-6}.

\begin{theorem}\label{thm-2Euclidean}
Under the Assumptions \ref{ass-3} and \ref{ass-6} with $\operatorname{E}\bm{X}=\bm{0}$, suppose the eigenvalues $\{\tau_{k}\}_{k\ge 1}$ of the operator $T$ decay as $\tau_{k}=t_{0}k^{-2r}$ for some $r, t_0 \in(0,\infty)$, then for $\rho\in\left(0,\frac{1}{8}\right)$, there exists a sequence $\{N_{n}\}_{n\ge1}$ satisfying
$$\log N_n \ge \left[\left(\frac{8 \left(t_0+\lambda_{\max}(D)\right) K_{\sigma^{2}} }{\rho \log2}\right)^{\frac{1}{1+2r}}n^{\frac{1}{1+2r}} +\left(\frac{8 \left(t_0+\lambda_{\max}(D)\right) K_{\sigma^{2}} }{\rho \log2}\right)^{\frac{2r}{1+2r}}n^{\frac{2r}{1+2r}}\right]\frac{\log2}{8}$$
such that when $n\ge\frac{\rho\log2}{8t_0K_{\sigma^{2}}}$ and $p=O(n^{\frac{2r}{1+2r}})<n$, the excess prediction risk satisfies
\begin{small}
\begin{equation*}
    \inf_{\tilde{\eta}}\sup_{\eta_0\in\mathbb{R}^p\times \mathcal{H}(K)}
    P\left(\mathcal{E}(\tilde{\eta})-\mathcal{E}(\eta_0)\ge
    \left(\frac{8\left( t_0 + \lambda_{\max}(D) \right)K_{\sigma^{2}}}{\rho\log2}\right)^{-\frac{2r}{1+2r}}\cdot\frac{2^{-2r}t_0 + \lambda_{\min}(D)}{2^{2r+3}n^{\frac{2r}{1+2r}}}\right)\ge \frac{\sqrt{N_n }}{1+\sqrt{N_n }}\left(1-4 \rho-\sqrt{\frac{4 \rho}{\log N_n }}\right),
\end{equation*}
\end{small}
where we identify the prediction rule $\tilde{\eta}$ as a arbitrary estimator $(\tilde{\bm{\alpha}},\tilde{\beta})$ based on the training samples $\{ (Z_{i},\bm{X}_{i},Y_{i}) \}_{i=1}^{n}$ ,and view $\eta_{0}$ as the true parameter $(\bm{\alpha}_{0},\beta_{0})\in\mathbb{R}^{p}\times\mathcal{H}(K)$. We emphasize the probability $P$ is taken over the product space of training samples $\{ (Z_{i},\bm{X}_{i},Y_{i}) \}_{i=1}^{n}$ generated by $\eta_{0}=(\bm{\alpha}_{0},\beta_{0})$.
\end{theorem}

In the existing literature, most results about the minimax bound are in the asymptotic sense, while the constants in our result are precise and specified. Letting $n \to \infty$ under the Assumption \ref{ass-8}, the lower bound inequality in the Theorem \ref{thm-2Euclidean}  implies as $N_n \to \infty$
\begin{equation*}
    \liminf\limits_{n\rightarrow\infty}\inf_{\tilde{\eta}}\sup_{\eta_{0}\in\mathbb{R}^{p}\times\mathcal{H}(K)}
    P(\mathcal{E}(\tilde{\eta})-\mathcal{E}(\eta_{0})\ge b^{\prime}\rho^{\frac{2r}{1+2r}}n^{-\frac{2r}{1+2r}})\ge 1-2\rho
\end{equation*}
for constant $b^{\prime}>0$, by which we get the asymptotic minimax lower bound:
$$ \lim\limits_{a\rightarrow0}\lim\limits_{n\rightarrow\infty}\inf_{\tilde{\eta}}\sup_{\eta_{0}\in\mathbb{R}^{p}\times\mathcal{H}(K)}
    P(\mathcal{E}(\tilde{\eta})-\mathcal{E}(\eta_{0})\ge an^{-\frac{2r}{1+2r}})=1.$$

\section{Proof of theorems and key Lemmas}
\subsection{Proof of the Theorem \ref{thm-1}}\label{sec-5}
\hspace{4mm} When setting $\lambda_{n}=\omega n^{-\frac{1}{1+\theta}}$, we have $ \frac{1}{n}<n^{-\frac{1}{1+\theta}}=\frac{\lambda_{n}}{\omega}=\frac{\sqrt{\lambda_{n}}}{\omega}\sqrt{\lambda_{n}}$
and
\begin{equation*}
    \frac{D(\lambda_{n})}{n}\le \frac{c(\omega n^{-\frac{1}{1+\theta}})^{-\theta}}{n}=c\omega^{-\theta}n^{-\frac{1}{1+\theta}}=c\omega^{-(1+\theta)}\lambda_{n},
\end{equation*}
by which we have
\begin{equation}\label{eq;BN}
B_{n}=\frac{1}{n\sqrt{\lambda_{n}}}+\sqrt{\frac{D(\lambda_{n})}{n}}\le\left(\omega^{-1}+\omega^{-\frac{1+\theta}{2}}\sqrt{c}\right)\sqrt{\lambda_{n}}
\end{equation}
 in Lemma \ref{3-1} and $n\lambda_{n}=\omega n^{\frac{\theta}{1+\theta}}\ge \omega$.

\par Applying Lemmas \ref{3-2}, \ref{3-5} and \ref{3-7} to \eqref{2-6}, when $n>n_{0}$, we have with probability at least $1-\delta_{2}-\delta_{4}-\delta_{5}$
\begin{align}
&~~~~\|\hatalpha{n}-\bm{\alpha}_{0}\|_2
    \le \opnorm{D_{n}^{-1}}\opnorm{G_{n}}\|\hatfn-f_{0}\|+\opnorm{D_{n}^{-1}}\|a_{n}\| \notag\\
    &\le \frac{3}{2}\opnorm{D^{-1}}\frac{\kappa M_2[v+8M_1{\log^{1/2} ({p}/{\delta_{2}})}]}{\sqrt{n}}\|\hatfn-f_{0} \|+\frac{3}{2}\opnorm{D^{-1}}\frac{c_{3}}{\sqrt{\delta_{4}}\sqrt{n}}=\frac{c_{4}}{\sqrt{n}}\|\hatfn-f_{0}
    \|+\frac{c_{5}}{\sqrt{n}},  \label{4-1}
\end{align}
where we let $c_{4}:={\kappa M_2[v+8M_1{\log^{1/2} ({p}/{\delta_{2}})}]}\frac{3\opnorm{D^{-1}}}{2}$ and $ c_{5}:=\frac{3c_{3}\opnorm{D^{-1}}}{2\sqrt{\delta_{4}}}$ with $c_{3}:=\sqrt{p}\sigma v$.
\par First, we give a crude bound for $\|\hatfn-f_{0}\|$ in \eqref{4-1}. According to \eqref{2-5}, it gives
\begin{align*}
    \|\hatfn-f_{0}\|
    &\le \lambda_{n}\opnorm{(T_{n}+\lambda_{n}I)^{-1}}\|f_{0}\|+\opnorm{(T_{n}+\lambda_{n}I)^{-1}}\opnorm{H_{n}}\|\hatalpha{n}-\bm{\alpha}_{0} \|+\|(T_{n}+\lambda_{n}I)^{-1}g_{n}\|  \\
    &=:I_{1}+I_{2}+I_{3}.
\end{align*}

For the term $I_{1}$, we have $I_{1}\le \|f_{0}\|$ because $\opnorm{(T_{n}+\lambda_{n}I)^{-1}}\le \frac{1}{\lambda_{n}}$.

For the term $I_{2}$, combining \eqref{4-1} we have with probability at least $1-\delta_{2}-\delta_{4}-\delta_{5}$
\begin{align*}
    I_{2}
    &\le \frac{1}{\lambda_{n}}\frac{c_{2}\log\left(\frac{2p}{\delta_{2}}\right)}{\sqrt{n}}\frac{c_{4}}{\sqrt{n}}\|\hatfn-f_{0}\|+\frac{1}{\lambda_{n}}\frac{c_{2}\log\left(\frac{2p}{\delta_{2}}\right)}{\sqrt{n}}\frac{c_{5}}{\sqrt{n}} \\
    &= \frac{c_{2}c_{4}\log\left(\frac{2p}{\delta_{2}}\right)}{n\lambda_{n}}\|\hatfn-f_{0}\|+\frac{c_{2}c_{5}\log\left(\frac{2p}{\delta_{2}}\right)}{n\lambda_{n}}\le \frac{c_{2}c_{4}\log\left(\frac{2p}{\delta_{2}}\right)}{n\lambda_{n}}\|\hatfn-f_{0}\|+\frac{c_{2}c_{5}}{\omega}\log\left(\frac{2p}{\delta_{2}}\right),~\text{by using}~n\lambda_{n}\ge \omega.
\end{align*}

For the term $I_{3}$, we obtain with probability at least $1-\delta_{3}$ by Lemma \ref{3-4} and $\opnorm{(T_{n}+\lambda_{n}I)^{-1}}\le \frac{1}{\lambda_{n}}$
\begin{equation*}
    I_{3}\le \opnorm{(T_{n}+\lambda_{n}I)^{-\frac{1}{2}}}
    \|(T_{n}+\lambda_{n}I)^{-\frac{1}{2}}g_{n}\|\le
    \frac{1}{\sqrt{\lambda_{n}}}\frac{\sigma}{\sqrt{\delta_{3}}}B_{n}\le \frac{\sigma\left(\omega^{-1}+\omega^{-\frac{1+\theta}{2}}\sqrt{c}\right)}{\sqrt{\delta_{3}}},
\end{equation*}
where we use \eqref{eq;BN} in the last step.

Thus we can bound $\|\hatfn-f_{0}\|$ by $I_{1}$, $I_{2}$ and $I_{3}$
\begin{align*}
    \|\hatfn-f_{0}\|
    &\le \|f_{0}\|+\frac{c_{2}c_{4}\log\left(\frac{2p}{\delta_{2}}\right)}{n\lambda_{n}}\|\hatfn-f_{0}\|+\frac{c_{2}c_{5}}{\omega}\log\left(\frac{2p}{\delta_{2}}\right)+\frac{\sigma\left(\omega^{-1}+\omega^{-\frac{1+\theta}{2}}\sqrt{c}\right)}{\sqrt{\delta_{3}}}\\
    &=\frac{c_{2}c_{4}\log\left(\frac{2p}{\delta_{2}}\right)}{n\lambda_{n}}\|\hatfn-f_{0}\|+c_{6},~\text{where}~c_{6}:=\|f_{0}\|+\frac{c_{2}c_{5}}{\omega}\log\left(\frac{2p}{\delta_{2}}\right)+\frac{\sigma\left(\omega^{-1}+\omega^{-\frac{1+\theta}{2}}\sqrt{c}\right)}{\sqrt{\delta_{3}}}.
\end{align*}

Notice when $\lambda_{n}=\omega n^{-\frac{1}{1+\theta}}$ and $n\lambda_{n}=\omega n^{\frac{\theta}{1+\theta}}$, and we have
\begin{equation*}
    \frac{c_{2}c_{4}\log\left(\frac{2p}{\delta_{2}}\right)}{n\lambda_{n}}\le \frac{1}{2}~\text{for}~n>\left(  \frac{2c_{2}c_{4}}{\omega}\log\left(\frac{2p}{\delta_{5}}\right) \right)^{\frac{1+\theta}{\theta}}.
\end{equation*}
Therefore, when $n>n_0$, we obtain with probability at least $1-\delta_{2}-\delta_{3}-\delta_{4}-\delta_{5}$
\begin{center}
$ \frac{1}{2}\|\hatfn-f_{0}\|
    \le \|\hatfn-f_{0}\|-\frac{c_{2}c_{4}\log\left(\frac{2p}{\delta_{2}}\right)}{n\lambda_{n}}\|\hatfn-f_{0}\| \le c_{6},$
which implies $\|\hatfn-f_{0}\|\le 2c_{6}$.
\end{center}
\par Second, we turn to bound $\|\hatalpha{n}-\bm{\alpha}_{0}\|$. By \eqref{4-1}, we have with probability at least $1-\delta_{2}-\delta_{3}-\delta_{4}-\delta_{5}$,
\begin{equation}\label{ALFA}
    \|\hatalpha{n}-\bm{\alpha}_{0}\|\le
    \frac{2c_{4}c_{6}+c_{5}}{\sqrt{n}}.
\end{equation}
\par Then we find a way to bound $\|T^{\frac{1}{2}}(\hatfn-f_{0})\|$.
According to \eqref{2-5}, we have
\begin{align*}
    \|T^{\frac{1}{2}}(\hatfn-f_{0})\|
    &\le \lambda_{n}\opnorm{T^{\frac{1}{2}}(T_{n}+\lambda_{n}I)^{-1}}\|f_{0}\|+\opnorm{T^{\frac{1}{2}}(T_{n}+\lambda_{n}I)^{-1}}\opnorm{H_{n}}\|\hatalpha{n}-\bm{\alpha}_{0}\|+\\
    &+\|T^{\frac{1}{2}}(T_{n}+\lambda_{n}I)^{-1}g_{n}\|=:E_{1}+E_{2}+E_{3}.
\end{align*}
\par For the term $E_{1}$, we have with probability at least $1-\delta_{1}$ by Inequality \ref{ineq-2} and Lemma \ref{3-1}
\begin{align}\label{eq:xxx}
    E_{1}
    &\le \lambda_{n}\opnorm{(T+\lambda_{n}I)^{\frac{1}{2}}(T_{n}+\lambda_{n}I)^{-\frac{1}{2}}}\opnorm{(T_{n}+\lambda_{n}I)^{-\frac{1}{2}}}\|f_{0}\| \nonumber \\
    &\le \lambda_{n}\left( \frac{1}{\sqrt{\lambda_{n}}}\opnorm{(T_{n}+\lambda_{n}I)^{-\frac{1}{2}}(T_{n}-T)}+1 \right)\frac{1}{\sqrt{\lambda_{n}}}\|f_{0}\|  \nonumber\\
    &\le \sqrt{\lambda_{n}}\|f_{0}\|\left( c_{1}\log\left(\frac{2}{\delta_{1}}\right)\frac{B_{n}}{\sqrt{\lambda_{n}}}  +1\right) \le \|f_{0}\|\left(c_{1}\left(\omega^{-1}+\omega^{-\frac{1+\theta}{2}}\sqrt{c}\right)\log\left(\frac{2}{\delta_{1}}\right)+1\right)\sqrt{\lambda_{n}}\\
    &=c_{7}\sqrt{\lambda_{n}},\nonumber
\end{align}
where we use the Inequality \ref{ineq-2} in the second inequality and the \eqref{eq;BN} in the last inequality
and define $c_{7}:=\|f_{0}\|\left(c_{1}\left(\omega^{-1}+\omega^{-\frac{1+\theta}{2}}\sqrt{c}\right)\log\left(\frac{2}{\delta_{1}}\right)+1\right).$

\par For the term $E_{3}$, by applying Inequality \ref{ineq-2} and Lemma \ref{3-4}, we have with probability at least $1-\delta_
{1}-\delta_{3}$
\begin{align*}
    E_{3}
    &\le \opnorm{(T+\lambda_{n}I)^{\frac{1}{2}}(T_{n}+\lambda_{n}I)^{-\frac{1}{2}}}\opnorm{(T_{n}+\lambda_{n}I)^{-\frac{1}{2}}(T+\lambda_{n}I)^{\frac{1}{2}}}\|(T+\lambda_{n}I)^{-\frac{1}{2}}g_{n}\| \\
    &\le \left( \frac{1}{\sqrt{\lambda_{n}}}\opnorm{(T_{n}+\lambda_{n}I)^{-\frac{1}{2}}(T_{n}-T)}+1 \right)^{2}\frac{\sigma}{\sqrt{\delta_{3}}}B_{n} \\
    &\le \left( c_{1}\log\left(\frac{2}{\delta_{1}}\right)\frac{B_{n}}{\sqrt{\lambda_{n}}}  +1\right)^{2}\frac{\sigma}{\sqrt{\delta_{3}}}B_{n}  \\
    &\le \left( c_{1}\left(\omega^{-1}+\omega^{-\frac{1+\theta}{2}}\sqrt{c}\right)\log\left(\frac{2}{\delta_{1}}\right)  +1\right)^{2}\frac{\sigma\left(\omega^{-1}+\omega^{-\frac{1+\theta}{2}}\sqrt{c}\right)}{\sqrt{\delta_{3}}}\sqrt{\lambda_{n}}=c_{8}\sqrt{\lambda_{n}},
\end{align*}
where in the last inequality $c_{8}:=\left[ c_{1}\left(\omega^{-1}+\omega^{-\frac{1+\theta}{2}}\sqrt{c}\right)\log\left(\frac{2}{\delta_{1}}\right)  +1\right]^{2}\frac{\sigma(\omega^{-1}+\omega^{-\frac{1+\theta}{2}}\sqrt{c})}{\sqrt{\delta_{3}}}.$

\par Notice we have \eqref{ALFA} with probability at least $1-\delta
_{2}-\delta_{3}-\delta_{4}-\delta_{5}$. Therefore, for the term $E_{2}$ we obtain with probability at least $1-\sum_{i=1}^{5}\delta_{i}$, by using $\opnorm{(T+\lambda_{n}I)^{\frac{1}{2}}(T_{n}+\lambda_{n}I)^{-\frac{1}{2}}}\le  (c_{1}(\coeffone)\log\left(\frac{2}{\delta_{1}}\right)  +1$ in \eqref{eq:xxx}, $\opnorm{(T_{n}+\lambda_{n}I)^{-\frac{1}{2}}}\le \frac{1}{\sqrt{\lambda_{n}}}$, Lemma \ref{3-2} and \eqref{ALFA},
\begin{align*}
    E_{2}
    &\le \opnorm{(T+\lambda_{n}I)^{\frac{1}{2}}(T_{n}+\lambda_{n}I)^{-\frac{1}{2}}}\opnorm{(T_{n}+\lambda_{n}I)^{-\frac{1}{2}}}\opnorm{H_{n}}\| \hatalpha{n}-\bm{\alpha}_{0}\| \\
    &\le \left[ (c_{1}(\coeffone)\log\left(\frac{2}{\delta_{1}}\right)  +1\right]\frac{1}{\sqrt{\lambda_{n}}}\cdot\frac{\kappa M_2[v+8M_1{\log^{1/2} ({p}/{\delta_{2}})}]}{\sqrt{n}}\cdot\frac{2c_{4}c_{6}+c_{5}}{\sqrt{n}} \\
    &\le \frac{\kappa M_2[v+8M_1{\log^{1/2} ({p}/{\delta_{2}})}](2c_{4}c_{6}+c_{5})}{\sqrt{\omega}}\left[ c_{1}(\coeffone)\log\left(\frac{2}{\delta_{1}}\right)  +1\right]\frac{1}{\sqrt{n}}=\frac{c_{9}}{\sqrt{n}},
\end{align*}
where in the last step we use $n\lambda_{n}\ge \omega$ and define
\begin{equation*}
    c_{9}:=\frac{\kappa M_2(2c_{4}c_{6}+c_{5})}{\sqrt{\omega}}\left[ c_{1}(\coeffone)\log\left(\frac{2}{\delta_{1}}\right)  +1\right]\left(v+8M_1{\log^{1/2} \left(\frac{p}{\delta_{2}}\right)}\right).
\end{equation*}

\par Thus, we bound $\|T^{\frac{1}{2}}(\hatfn-f_{0})\|$ by
\begin{equation*}
    \|T^{\frac{1}{2}}(\hatfn-f_{0})\|
    \le E_{1}+E_{2}+E_{3} \le (c_{7}+c_{8})\sqrt{\lambda_{n}}+\frac{c_{9}}{\sqrt{n}}.
\end{equation*}
Recall the excess prediction risk can be bounded by
\begin{equation*}
\mathcal{E}(\hat{\eta}_{n})-\mathcal{E}(\eta_{0})
\le 2\lambda_{\max}\|\hatalpha{n}-\bm{\alpha}_{0}\|_2^{2}+2\|T^{\frac{1}{2}}(\hat{f}_{n}-f_{0})\|^{2},
\end{equation*}
based on which we further have
\begin{align*}
    \mathcal{E}(\hat{\eta}_{n})-\mathcal{E}(\eta_{0})
    &\le 2\lambda_{\max}\left( \frac{2c_{4}c_{6}+c_{5}}{\sqrt{n}} \right)^{2}+ 2\left( (c_{7}+c_{8})\sqrt{\lambda_{n}}+\frac{c_{9}}{\sqrt{n}} \right)^{2}=C_{1}n^{-1} +C_{2} \sqrt{\frac{\lambda_{n}}{n}} +C_{3}\lambda_{n},
\end{align*}
where $C_{1}:=2\lambda_{\max}(2c_{4}c_{6}+c_{5})^{2}+2c_{9}^{2},~
    C_{2}:=4(c_{7}+c_{8})c_{9}~\text{and}~
    C_{3}:=2(c_{7}+c_{8})^{2}.$

Finally we get the desired conclusion after we notice $\sqrt{\frac{\lambda_{n}}{n}}=\sqrt{\omega}n^{-\frac{2+\theta}{2+2\theta}}$ in the above proof. $\hfill\square$

\subsection{Proof of the Theorem \ref{thm-2Euclidean}}\label{sec-6}
{The whole proof of Theorem \ref{thm-2Euclidean} is served for the condition where $p$ has increasing number of dimension.}
\hspace{4mm} Let $M$ be the smallest integer greater than $b_{0}n^{\frac{1}{1+2r}}$ and $L=M^{2r}$, where constant $b_{0}$ will be defined in later proof. For two binary sequences $\gamma=(\gamma_{L+1},\cdots,\gamma_{2L})\in \{0,1 \}^{L}$ and $\theta=(\theta_{M+1},\cdots,\theta_{2M})\in \{0,1 \}^{M}$, define
\begin{equation*}
\beta_{\theta}=M^{-\frac{1}{2}}\sum\limits_{k=M+1}^{2M}\theta_{k}\K \varphi_{k} \quad \text{and} \quad \bm{\alpha}_{\gamma}=\frac{1}{L}(\gamma_{L+1},\cdots,\gamma_{2L}).
\end{equation*}
By applying $\innerproduct{\K\varphi_{j}}{\K\varphi_{k}}_{K}=\innerproduct{\varphi_{j}}{L_{K}\varphi_{k}}_{K}=\innerproduct{\varphi_{j}}{\varphi_{k}}=\delta_{j k}$, we have $\|\bm{\alpha}_{\gamma}\|_1 \le 1$ and  $\beta_{\theta}\in \mathcal{H}(K)$ noticing
\begin{align*}
    \left\|\beta_{\theta}\right\|_{K}^{2}
    &=\left\|M^{-\frac{1}{2}}\sum\limits_{k=M+1}^{2M}\theta_{k}\K \varphi_{k}\right\|_{K}^{2}  =\sum\limits_{k=M+1}^{2M}M^{-1}\theta_{k}^{2}\left\|\K\varphi_{k}\right\|_{K}^{2} \le M^{-1}\sum\limits_{k=M+1}^{2M} \left\|\K\varphi_{k}\right\|_{K}^{2} =1.
\end{align*}

Using the Lemma \ref{VGlemma}, there exist sets $\Gamma=\{\gamma^{i}\}_{i=0}^{N_\gamma}\subset\{0,1\}^{L}$ and $\Theta=\{\theta^{i}\}_{i=0}^{N_\theta}\subset\{0,1\}^{M}$ such that
\begin{center}
\rm{(ia)} $\gamma^{0}=(0,\cdots,0)$,~~
\rm{(iia)} $H(\gamma^{i}, \gamma^{j}) > \frac{L}{8}$ for all $i \neq j$,~~
\rm{(iiia)} $N_\gamma \geq 2^{\frac{L}{8}}$,
\end{center}
\begin{center}
\rm{(ib)} $\theta^{0}=(0,\cdots,0)$,~~
\rm{(iib)} $H(\theta^{i}, \theta^{j}) > \frac{M}{8}$ for all $i \neq j$,~~
\rm{(iiib)} $N_\theta \geq 2^{\frac{M}{8}}$,
\end{center}
where $H(\theta,\theta^{\prime})$ is the Hamming distance between $\theta$ and $\theta^{\prime}$. Define the combined parameter $$\kappa=(\gamma_{L+1},\cdots,\gamma_{2L},\theta_{M+1},\cdots,\theta_{2M})\in \{0,1 \}^{L+M}.$$ The construction of $\Gamma$ and $\Theta$ implies that there exists a set $\Theta=\{\kappa^{i}\}_{i=0}^{N_\kappa} \subset\{0,1\}^{L+M}$ such that
\begin{center}
\rm{(ic)} $\kappa^{0}=(0, \cdots, 0)$,~~
\rm{(iic)} $H(\kappa^{i}, \kappa^{j})>\frac{L+M}{8}$ for all $i \neq j$,~~
\rm{(iiic)} $N_\kappa \geq 2^{\frac{L+M}{8}}$.
\end{center}

Let $P^{n}_{\bm{\alpha}_{0},\beta_{0}}$ be the joint distribution on the product space of training samples $\{(Z_{i},\bm{X}_{i},Y_{i})\}_{i=1}^{n}$ generated by the true parameter $(\bm{\alpha}_{0},\beta_{0})$, where $Z_{i}=\bm{X}_{i}^{T}\bm{\alpha}_{0}+\innerproduct{Y_{i}}{\beta_{0}}+\varepsilon_{i}$, and $P_{\bm{\alpha}_{0},\beta_{0}}$ be the distribution on a single sample $(Z,\bm{X},Y)$, where $Z=\bm{X}^{T}\bm{\alpha}_{0}+\innerproduct{Y}{\beta_{0}}+\varepsilon$.
\par By the independence of the training samples, for different $\theta,\theta^{\prime}\in\Theta$ and $\gamma,\gamma^{\prime}\in\Gamma$, we have
\begin{equation*}
    \log\left(\frac{\mathrm{d}P^{n}_{\bm{\alpha}_{\gamma^{\prime}},\beta_{\theta^{\prime}}}}{\mathrm{d}P^{n}_{\bm{\alpha}_{\gamma},\beta_{\theta}}}(\{(Z_{i},\bm{X}_{i},Y_{i})\}_{i=1}^n)\right)=\sum\limits_{i=1}^{n}\log\left( \frac{\mathrm{d}P_{\bm{\alpha}_{\gamma^{\prime}},\beta_{\theta^{\prime}}}}{\mathrm{d}P_{\bm{\alpha}_{\gamma},\beta_{\theta}}}(Z_{i},\bm{X}_{i},Y_{i}) \right).
\end{equation*}
Using the Assumption \ref{ass-6}, we can bound the Kullback-Leibler distance between $P^{n}_{\bm{\alpha}_{\gamma^{\prime}},\beta_{\theta^{\prime}}}$ and $P^{n}_{\bm{\alpha}_{\gamma},\beta_{\theta}}$
\begin{align*}
    K(P^{n}_{\bm{\alpha}_{\gamma^{\prime}},\beta_{\theta^{\prime}}} | P^{n}_{\bm{\alpha}_{\gamma},\beta_{\theta}})
    &=\sum\limits_{i=1}^{n}\operatorname{E}_{\bm{\alpha}_{\gamma^{\prime}},\beta_{\theta^{\prime}}}\left(\log\left( \frac{\mathrm{d}P_{\bm{\alpha}_{\gamma^{\prime}},\beta_{\theta^{\prime}}}}{\mathrm{d}P_{\bm{\alpha}_{\gamma},\beta_{\theta}}}\right)\right)\le n K_{\sigma^{2}} \operatorname{E}[\innerproduct{Y}{\beta_{\theta^{\prime}}-\beta_{\theta}}+\bm{X}^{T}(\bm{\alpha}_{\gamma^{\prime}}-\bm{\alpha}_{\gamma})]^{2}\\
    & \le 2n K_{\sigma^{2}} \operatorname{E}(\innerproduct{Y}{\beta_{\theta^{\prime}}-\beta_{\theta}})^2+2n K_{\sigma^{2}} \operatorname{E}[\bm{X}^{T}(\bm{\alpha}_{\gamma^{\prime}}-\bm{\alpha}_{\gamma})]^{2}.
\end{align*}
Noticing $\innerproduct{\K\varphi_{j}}{L_{C}\K\varphi_{k}}=\innerproduct{\varphi_{j}}{T\varphi_{k}}=\tau_{k}\delta_{j k}$, we have
\begin{align*}
    \operatorname{E}(\innerproduct{Y}{\beta_{\theta^{\prime}}-\beta_{\theta}})^{2}
    &=\innerproduct{\beta_{\theta^{\prime}}-\beta_{\theta}}{L_{C}(\beta_{\theta^{\prime}}-\beta_{\theta})} \\
    &=\left\langle  M^{-\frac{1}{2}}\sum\limits_{k=M+1}^{2M}(\theta^{\prime}_{k}-\theta_{k})\K \varphi_{k} , M^{-\frac{1}{2}}\sum\limits_{k=M+1}^{2M}(\theta^{\prime}_{k}-\theta_{k})L_{C}\K \varphi_{k}   \right\rangle \\
    &=M^{-1} \sum\limits_{k=M+1}^{2M} (\theta^{\prime}_{k}-\theta_{k})^{2}\tau_{k} \le M^{-1}\tau_{M}\sum\limits_{k=M+1}^{2M} (\theta^{\prime}_{k}-\theta_{k})^{2} =M^{-1}\tau_{M}H(\theta^{\prime},\theta)\le \tau_{M}= t_0M^{-2r},
\end{align*}
where the last inequality is by the eigen-decay condition. For parametric part, we have
\begin{center}
$\operatorname{E}[\bm{X}^{T}(\bm{\alpha}_{\gamma^{\prime}}-\bm{\alpha}_{\gamma})]^{2}\le \lambda_{\max}(D)\|\bm{\alpha}_{\gamma^{\prime}}-\bm{\alpha}_{\gamma}\|_2^{2}=\lambda_{\max}(D) H({\gamma^{\prime}}, \gamma)/L^2 \le \lambda_{\max}(D)/L=\lambda_{\max}(D)M^{-2r}$,
\end{center}
 from which we obtain
\begin{center}
$ K(P^{n}_{\bm{\alpha}_{\gamma^{\prime}},\beta_{\theta^{\prime}}} | P^{n}_{\bm{\alpha}_{\gamma},\beta_{\theta}})\le 2\left(t_0+\lambda_{\max}(D)\right) n K_{\sigma^{2}} M^{-2r}$.
\end{center}

If we put $b_{0}:=\left(\frac{8 \left(t_0+\lambda_{\max}(D)\right) K_{\sigma^{2}} }{\rho \log2}\right)^{\frac{1}{1+2r}} $, then for any $\rho\in(0,\frac{1}{8})$, we have by $M \ge b_{0}n^{\frac{1}{1+2r}}$
\begin{equation*}
    \frac{1}{N_\kappa}\sum\limits_{j=1}^{N_\kappa} K(P_{\alpha_{\gamma^{j}},\beta_{\theta^{j}}} | P_{{\alpha}_{\gamma^{0}},\beta_{\theta^{0}}})\le 2\left(t_0+\lambda_{\max}(D)\right) n K_{\sigma^{2}} M^{-2r}\le 2 \rho \log\left(2^{\frac{M}{8}}\right) \le 2 \rho \log\left(2^{\frac{L+M}{8}}\right)\le 2 \rho\log N_\kappa.
\end{equation*}

For $\theta\in\Theta$ and and $\gamma\in\Gamma$, let the prediction rule $\eta_{\gamma,\theta}$ be $\eta_{\gamma,\theta}(\bm{X},Y):=\bm{X}^{T}\bm{\alpha}_{\gamma}+\innerproduct{Y}{\beta_{\theta}}$. For different $\theta,\theta^{\prime}\in\Theta$  and $\gamma,\gamma^{\prime}\in\Gamma$, when the true parameter is $(\bm{\alpha}_{\gamma},\beta_{\theta})$, the excess prediction risk for the prediction rule $\eta_{\gamma^{\prime},\theta^{\prime}}$ is
\begin{align*}
    \mathcal{E}(\eta_{\gamma^{\prime},\theta^{\prime}})-\mathcal{E}(\eta_{\gamma,\theta})
    &=\operatorname{E}[\innerproduct{Y}{\beta_{\theta^{\prime}}-\beta_{\theta}}+\bm{X}^{T}(\bm{\alpha}_{\gamma^{\prime}}-\bm{\alpha}_{\gamma})]^{2} \\
    &=\operatorname{E}(\innerproduct{Y}{\beta_{\theta^{\prime}}-\beta_{\theta}})^2 + 2 \operatorname{E}\left(\innerproduct{Y}{\beta_{\theta^{\prime}}-\beta_{\theta}} \cdot \bm{X}^{T}(\bm{\alpha}_{\gamma^{\prime}}-\bm{\alpha}_{\gamma})\right) + \operatorname{E}[\bm{X}^{T}(\bm{\alpha}_{\gamma^{\prime}}-\bm{\alpha}_{\gamma})]^2 \\
   [\text{By}~\operatorname{E}\bm{X}=\bm{0}~\text{and}~\bm{X}\perp~Y]~ &= M^{-1}\sum\limits_{k=M+1}^{2M}(\theta^{\prime}_{k}-\theta_{k})^{2}\tau_{k} + (\bm{\alpha}_{\gamma^{\prime}}-\bm{\alpha}_{\gamma})   ^T\operatorname{E}[\bm{X}\bm{X}^{T}](\bm{\alpha}_{\gamma^{\prime}}-\bm{\alpha}_{\gamma}) \\
    &\ge M^{-1}\tau_{2M}\sum\limits_{k=M+1}^{2M}(\theta^{\prime}_{k}-\theta_{k})^{2}+ \lambda_{\min}(D) \|\bm{\alpha}_{\gamma^{\prime}}-\bm{\alpha}_{\gamma}\|_2^{2} \\
    & = M^{-1}\tau_{2M}H(\theta^{\prime},\theta)+ L^{-2} \lambda_{\min}(D)H({\gamma^{\prime}}, \gamma) \\
    &\ge t_0 M^{-1}(2M)^{-2r}\frac{M}{8}+L^{-2}\lambda_{\min}(D)\frac{L}{8} = \frac{1}{8}(2^{-2r}t_0+\lambda_{\min}(D))M^{-2r}.
\end{align*}
Notice $M$ is the smallest integer greater than $b_{0}n^{\frac{1}{1+2r}}$, and to control the lower bound of last expression we additionally set $M\le 2b_{0}n^{\frac{1}{1+2r}}$. Thus
\begin{center}
$b_{0}n^{\frac{1}{1+2r}}\ge 1\Leftrightarrow n\ge \frac{\rho\log2}{8\left( t_0 + \lambda_{\max}(D) \right)K_{\sigma^{2}}}$.
\end{center}
 Therefore, we obtain the lower bound for $\mathcal{E}(\eta_{\gamma^{\prime},\theta^{\prime}})-\mathcal{E}(\eta_{\gamma,\theta})$
\begin{align*}
   \mathcal{E}(\eta_{\gamma^{\prime},\theta^{\prime}})-\mathcal{E}(\eta_{\gamma,\theta})&\ge 2^{-3}(2^{-2r}t_0+\lambda_{\min}(D))(2 b_{0} n^{\frac{1}{1+2r}})^{-2r} \\
   &=2^{-(2r+3)}\left( 2^{-2r}t_0 + \lambda_{\min}(D) \right)\left(\frac{8\left( t_0 + \lambda_{\max}(D) \right)K_{\sigma^{2}}}{\rho\log2}\right)^{-\frac{2r}{1+2r}}n^{-\frac{2r}{1+2r}}.
\end{align*}

\par Consider the set $\Xi:=\{(\bm{\alpha}_{\gamma},\beta_{\theta}):\gamma\in \Gamma , \theta\in\Theta\}$.
By the Lemma \ref{KLlemma}, we have
\begin{equation*}
    \inf_{\tilde{\eta}}\sup_{\eta_0\in\Xi}
    P\left(\mathcal{E}(\tilde{\eta})-\mathcal{E}(\eta_0)\ge\left(\frac{8\left( t_0 + \lambda_{\max}(D) \right)K_{\sigma^{2}}}{\rho\log2}\right)^{-\frac{2r}{1+2r}}\cdot\frac{2^{-2r}t_0 + \lambda_{\min}(D)}{2^{2r+3}n^{\frac{2r}{1+2r}}}\right)\ge \frac{\sqrt{N_\kappa }}{1+\sqrt{N_\kappa }}\left(1-4 \rho-\sqrt{\frac{4 \rho}{\log N_\kappa }}\right).
\end{equation*}

From the construction of $\bm{\alpha}_\gamma\in \mathbb{R}^p$, we see $p=L=M^{2r}\asymp n^{\frac{2r}{1+2r}}$. Note that for fixed $\tilde{\eta}$, we have
$\sup\limits_{\eta_0\in\mathbb{R}^p\times\mathcal{H}(K)}
P\left\{\mathcal{E}(\tilde{\eta})-\mathcal{E}(\eta_0)\ge\cdots\right\} \ge \sup\limits_{\eta_0\in\Xi}P\left\{\mathcal{E}(\tilde{\eta})-\mathcal{E}(\eta_0)\ge\cdots\right\}
$
and by \rm{(iiic)} we get
$$\log N_\kappa\ge \frac{M+L}{8}\log2=\left\{\left(\frac{8 \left(t_0+\lambda_{\max}(D)\right) K_{\sigma^{2}} }{\rho \log2}\right)^{\frac{1}{1+2r}}n^{\frac{1}{1+2r}} +\left[\frac{8 \left(t_0+\lambda_{\max}(D)\right) K_{\sigma^{2}} }{\rho \log2}\right]^{\frac{2r}{1+2r}}n^{\frac{2r}{1+2r}}\right\}\frac{\log2}{8}.$$

We obtain the desired conclusion.
$\hfill\square$

\subsection{Proofs of the key Lemmas}\label{sec-7}

\subsubsection{The Derivation of \eqref{2-5} and \eqref{2-6}}\label{App-A}
\hspace{4mm}Recall that
$Z_{i}=\bm{X}_{i}^{T}\bm{\alpha}_{0}+\langle Y_{i},\K f_{0} \rangle+\varepsilon_{i}$, where $\{(\bm{X}_{i},Y_{i},\varepsilon_{i})\}_{i=1}^n$ are independent copies of $(\bm{X},Y,\varepsilon)$ in \eqref{2-1}. Thus the right side of \eqref{2-4} can be written as
\begin{equation*}
    F_{n}(\bm{\alpha},f):=\frac{1}{n}\sum\limits_{i=1}^{n}\left( \bm{X}^{T}_{i}(\bm{\alpha}-\bm{\alpha}_{0})+\langle Y_{i},L^{\frac{1}{2}}_{K}(f-f_{0}) \rangle -\varepsilon_{i} \right)^{2}+\lambda_{n}\|f\|^{2}.
\end{equation*}
\par Notice $(\hatalpha{n},\hatfn)$ is the minimum of $F_{n}(\bm{\alpha},f)$, therefore $\frac{\partial F_{n}(\hatalpha{n},\hatfn)}{\partial \bm{\alpha}}=0,$ from which we have
\begin{align*}
0
&=\frac{1}{n}\sum\limits_{i=1}^{n}\bm{X}_{i}\left( \bm{X}^{T}_{i}(\hat{\bm{\alpha}}_{n}-\bm{\alpha}_{0})+\langle Y_{i},L^{\frac{1}{2}}_{K}(\hatfn-f_{0}) \rangle -\varepsilon_{i} \right)  \\
&=\left(\frac{1}{n}\sum\limits_{i=1}^{n}\bm{X}_{i}\bm{X}_{i}^{T}\right)(\hat{\bm{\alpha}}_{n}-\bm{\alpha}_{0})+\frac{1}{n}\sumn \langle Y_{i},L^{\frac{1}{2}}_{K}(\hatfn-f_{0}) \rangle \bm{X}_{i} -\frac{1}{n}\sumn\varepsilon_{i}\bm{X}_{i}:=D_{n}(\hat{\bm{\alpha}}_{n}-\bm{\alpha}_{0})+G_{n}(\hatfn-f_{0})-a_{n},
\end{align*}
thus $\hat{\bm{\alpha}}_{n}-\bm{\alpha}_{0}=-D_{n}^{-1}G_{n}(\hatfn-f_{0})+D_{n}^{-1}a_{n}$.
\par Next define the function $\varphi_{n}(t;\bm{\alpha},f,g):=F_{n}(\bm{\alpha},f+t g)$, and the fact that $(\hatalpha{n},\hatfn)$ minimizes $F_{n}(\bm{\alpha},f)$ implies
\begin{equation*}
    \frac{\mathrm{d}\varphi(t;\hatalpha{n},\hatfn,g)}{\mathrm{d}t}\bigg|_{t=0},\quad \forall g\in L^{2}(\mathcal{T}),
\end{equation*}
from which we have
\begin{align}\label{A-1}
 0
&=\frac{1}{n}\sum\limits_{i=1}^{n}\left( \bm{X}^{T}_{i}(\hatalpha{n}-\bm{\alpha}_{0})+\langle Y_{i},L^{\frac{1}{2}}_{K}(\hatfn-f_{0}) \rangle -\varepsilon_{i} \right)\langle Y_{i},\K g \rangle+\lambda_{n}\langle \hatfn,g \rangle.
\end{align}
From \eqref{A-1}, we have
\begin{equation}\label{A-2}
\frac{1}{n}\sumn\bm{X}^{T}_{i}(\hatalpha{n}-\bm{\alpha}_{0})\K Y_{i}+\frac{1}{n}\sumn\langle Y_{i},L^{\frac{1}{2}}_{K}(\hatfn-f_{0}) \rangle\K Y_{i}-\frac{1}{n}\sumn\varepsilon_{i}\K Y_{i}+\lambda_{n}\hatfn=0.
\end{equation}
Notice $ L_{C_{n}}f
    =\left\langle \frac{1}{n}\sumn Y_{i}(s)Y_{i}(t),f(t) \right\rangle
    =\frac{1}{n}\sumn \langle Y_{i},f \rangle Y_{i}$ and recall that $a_{n}:=\frac{1}{n}\sum_{i=1}^{n}\varepsilon_{i}\bm{X}_{i}$, $H_{n}(\bm{\alpha}):=\frac{1}{n}\sum\limits_{i=1}^{n}(\bm{X}_{i}^{T}\bm{\alpha}) \K Y_{i}$ and $T_{n}=\K\circ L_{C_{n}}\circ\K$. Thus \eqref{A-2} can be reformulated as
\begin{equation*}
    (T_{n}+\lambda_{n}I)\hatfn-T_{n}f_{0}+H_{n}(\hatalpha{n}-\bm{\alpha}_{0})-g_{n}=0,
\end{equation*}
which yields
\begin{align*}
\hatfn-f_{0}
&=(T_{n}+\lambda_{n}I)^{-1}T_{n}f_{0}-f_{0}-(T_{n}+\lambda_{n}I)^{-1}H_{n}(\hatalpha{n}-\bm{\alpha}_{0})+(T_{n}+\lambda_{n}I)^{-1}g_{n}\\
&=-\lambda_{n}(T_{n}+\lambda_{n}I)^{-1}f_{0}-(T_{n}+\lambda_{n}I)^{-1}H_{n}(\hatalpha{n}-\bm{\alpha}_{0})+(T_{n}+\lambda_{n}I)^{-1}g_{n}.
\end{align*} $\hfill\square$

\subsubsection{Proof of Lemma \ref{3-2}}
\hspace{4mm}We first prove $H_{n}=G_{n}^{*}$, which shows $\|G_{n}\|_{\operatorname{op}}=\|H_{n}\|_{\operatorname{op}}$.
For any $\bm{\gamma}\in \mathbb{R}^{p}$ and $f\in L^{2}(\mathcal{T})$, one has
\begin{equation*}
    \bm{\gamma}^{T}G_{n}(f)
    =\frac{1}{n}\sumn\innerproduct{Y_{i}}{\K f}(\bm{\gamma}^{T}\bm{X}_{i})=\frac{1}{n}\sumn(\bm{X}_{i}^{T}\bm{\gamma})\innerproduct{\K Y_{i}}{f}=\innerproduct{H_{n}(\bm{\gamma})}{f}.
\end{equation*}
\par Now we turn to bound $\|G_{n}\|_{\operatorname{op}}$, for $1\le j\le p$, define the operator $G_{n,j}:L^{2}(\mathcal{T})\mapsto\mathbb{R}$ by
\begin{equation*}
    G_{n,j}(f):=\frac{1}{n}\sumn\innerproduct{Y_{i}}{\K f}X_{i,j},
\end{equation*}
which can also be viewed as a random variable taking values in a Hilbert space $L^{2}(\mathcal{T})^{*}=L^{2}(\mathcal{T})$.
\par Notice $G_{n}=(G_{n,1},\cdots,G_{n,p})$, thus we have
\begin{equation*}
    \|G_{n}(f)\|^{2}
    =\sum\limits_{j=1}^{p}|G_{n,j}(f)|^{2}
    \le \sum\limits_{j=1}^{p} \|G_{n,j}\|_{\operatorname{op}}^{2}\|f\|^{2}
    \le \left( \sum\limits_{j=1}^{p} \|G_{n,j}\|_{\operatorname{op}} \right)^{2}\| f \|^{2},
\end{equation*}
which means $\opnorm{G_{n}}\le \sum\limits_{j=1}^{p}\opnorm{G_{n,j}}$. Define the operator $\xi_{i,j}:L^{2}(\mathcal{T})\mapsto\mathbb{R}$ and $\xi_{j}:L^{2}(\mathcal{T})\mapsto\mathbb{R}$ by
\begin{equation*}
    \xi_{i,j}(f):=\innerproduct{Y_{i}}{\K f}X_{i,j}
    \quad\text{and}\quad
    \xi_{j}(f):=\innerproduct{Y}{\K f}X_{j},
\end{equation*}
from which we rewrite $G_{n,j}=\frac{1}{n}\sumn \xi_{i,j}$, where $\{\xi_{i,j}\}_{i=1}^{n}$ are independent copies of $\xi_{j}$.

{{By the definition of $\xi_{j}$, after noticing the isomorphism between $L^{2}(\mathcal{T})^{*}$ and $L^{2}(\mathcal{T})$, we have
\begin{equation*}
    \|\xi_{j}\|_{\operatorname{op}} = \|X_{j}\K Y\| \le |X_{j}|\cdot\opnorm{\K}\cdot\|Y\| \le \kappa M_2|X_{j}|
\end{equation*}
by $\opnorm{\K} = \kappa$ in Assumption \ref{ass-0} and Assumption \ref{ass-2}. After taking expectation and using the sub-Gaussian growth of moments condition for $X_{j}$, we have
\begin{equation*}
{\left\| \|\xi_{j}\|_{\operatorname{op}} \right\|_{G}}=\sup_{k \ge 1} \left[ {\frac{{\rm{E}}{\|\xi_{j}\|_{\operatorname{op}}^{2k}}}{(2 k-1) ! !}} \right]^{1/(2k)}   \le  \kappa M_2\sup_{k \ge 1} \left[ {\frac{{\rm{E}}{|X_{j}|^{2k}}}{(2 k-1) ! !}} \right]^{1/(2k)}= \kappa M_2{\left\| |X_{j}| \right\|_{G}}\le  \kappa M_2M_1.
\end{equation*}
Because $\bm{X}$ and $Y$ are independent with zero mean, we have
\begin{center}
$(\operatorname{E}\xi_{j})(f)=
    \innerproduct{\operatorname{E}Y_{i}}{\K f}\cdot\operatorname{E}X_{j}=0,$ which means $\operatorname{E}\xi_{j}=0$.
\end{center}
\par By using Corollary \ref{Lip.Class.lem}, we have
for any $\delta_{2}\in(0,1)$, with probability at least $1-\frac{\delta_{2}}{p}$,
\begin{align*}
    \opnorm{G_{n,j}}
   =\opnorm{\frac{1}{n}\sumn\xi_{i,j}}&\le \frac{1}{\sqrt n}\left\{\sqrt{\frac{1}{n} \sum\limits_{i=1}^n{\E \| \xi_{i,j}\|_{\operatorname{op}}^2}}+8\sqrt{\frac{1}{n}\sum\limits_{i=1}^n \| \| \xi_{i,j}\|_{\operatorname{op}}\|_{G}^2 \log (\frac{p}{\delta_{2}})}\right\}\\
[\text{By Assumptions \ref{ass-1}}]~     &\le \frac{\kappa M_2[v+8M_1{\log^{1/2} ({p}/{\delta_{2}})}]}{\sqrt n}.
\end{align*}

\par Notice $\opnorm{G_{n}}\le\sum\limits_{j=1}^{p} \opnorm{G_{n,j}}$, thus we have the following relation for the events
\begin{equation*}
    \left\{ \opnorm{G_{n}} \ge  \frac{\kappa M_2[v+8M_1{\log^{1/2} ({p}/{\delta_{2}})}]}{\sqrt n} \right\} \subset \bigcup_{1\le j\le p}
    \left\{  \opnorm{G_{n,j}} \ge \frac{\kappa M_2[v+8M_1{\log^{1/2} ({p}/{\delta_{2}})}]}{\sqrt n} \right\},
\end{equation*}
from which we have
\begin{align*}
    P\left(\opnorm{G_{n}} \ge  \frac{\kappa M_2[v+8M_1{\log^{1/2} ({p}/{\delta_{2}})}]}{\sqrt n}\right) & \le \sum\limits_{j=1}^{p}P\left(\opnorm{G_{n,j}} \ge \frac{\kappa M_2[v+8M_1{\log^{1/2} ({p}/{\delta_{2}})}]}{\sqrt n} \right)\le\sum\limits_{j=1}^{p}\frac{\delta_{2}}{p}=\delta_{2}.
\end{align*}
\par Therefore, $P\left(\opnorm{G_{n}}\le \frac{\kappa M_2[v+8M_1{\log^{1/2} ({p}/{\delta_{2}})}]}{\sqrt n}\right)\ge 1-\delta_{2}.$
}}
\subsubsection{Lemma \ref{3-5}}
\hspace{4mm}The Lemma \ref{3-5} shows the fact $\|a_{n}\|=O_{p}(n^{-\frac{1}{2}})$, and its proof is based on the Markov's inequality.
\begin{lemma}\label{3-5}
Under the Assumptions \ref{ass-1} and \ref{ass-3}, for any $\delta_{4}\in(0,1)$, with probability at least $1-\delta_{4}$, we have
\begin{equation*}
    \|a_{n}\|\le \frac{c_{3}}{\sqrt{\delta_{4}}\sqrt{n}}
    \quad\text{with}\quad c_{3}:=\sqrt{p}\sigma v.
\end{equation*}
\end{lemma}
\begin{proof}
Define $\eta:=\varepsilon \bm{X}$ and $\eta_{i}:=\varepsilon_{i}\bm{X}_{i}\,(1\le i\le n)$, by which we rewrite $a_{n}=\frac{1}{n}\sumn \xi_{i}$.
\par {Notice for $i\neq j$, we have $\operatorname{E}(\eta_{i}^{T}\eta_{j})=\operatorname{E}(\varepsilon_{i}\varepsilon_{j}\bm{X}_{i}^{T}\bm{X}_{j})=\operatorname{E}(\varepsilon_j \bm{X}^T_i \bm{X}_j\operatorname{E}(\varepsilon_i|\bm{X}_i))=0$. And
\begin{equation*}
    \operatorname{E}(\| \eta\|^{2})=\operatorname{E}(\varepsilon^{2}\|\bm{X}\|^{2})=\operatorname{E}(\|\bm{X}\|^2\operatorname{E(\varepsilon^2|\bm{X})}) \le \sigma^{2}\sum\limits_{j=1}^{p}\operatorname{E}|X_{j}|^{2}\le p\sigma^{2}v^{2},
\end{equation*}}
where we use the second moment condition of $\bm{X}$ in Assumption \ref{ass-1}.
Therefore, we have
\begin{equation*}
    \operatorname{E}\left(\|a_{n}\|^{2}\right)=\operatorname{E}\left(\left\|\frac{1}{n}\sumn\eta_{i}\right\|^{2}\right)=\frac{\operatorname{E}\left(\|\eta_1\|^{2}\right)}{n}\le \frac{p\sigma^{2}v^{2}}{n},
\end{equation*}
from which we obtain the Markov's inequality for $a_{n}$: $P\left(\|a_{n}\|\ge t\right)\le \frac{p\sigma^{2}v^{2}}{n t^{2}}.$ Therefore, we can conclude for $\delta_{4}\in(0,1)$, we have with probability at least $1-\delta_{4}$
\begin{equation*}
    \|a_{n}\|\le \frac{c_{3}}{\sqrt{\delta_{4}}\sqrt{n}} \quad\text{with}\quad c_{3}:= \sqrt{p}\sigma v.
\end{equation*}
\end{proof}

\subsubsection{Lemma \ref{3-6}}
\hspace{4mm}The Lemma \ref{3-6} is the concentration inequality for the empirical covariance matrix $D_{n}$.

\begin{lemma}\label{3-6} Under the Assumption \ref{ass-1}, for $t>0$, {we have
\begin{equation*}
    P(\|D_{n}-D\|_{\infty}\ge t)\le 2p^{2}\exp\left( -\frac{n t^{2}}{8(64 M_{1}^{4}+M_1^2 t)} \right).
\end{equation*}}
\end{lemma}
\begin{proof}
For $1\le j,k \le p$, put
\begin{equation*}
    d_{j k}:=(D)_{j k}
    \quad\text{and}\quad
    d^{n}_{j k}:=(D_{n})_{j k}=\frac{1}{n}\sumn X_{i,j}X_{i,k}=\frac{1}{n}\sumn d^{n}_{j k,i} \quad \text{with}\quad d^{n}_{j k,i}:=X_{i,j}X_{i,k}.
\end{equation*}
Let $e^{n}_{j k,i}$ be the centralization of $d^{n}_{j k,i}$, i.e. $e^{n}_{j k,i}:=d^{n}_{j k,i}-\operatorname{E}d^{n}_{j k,i}=d^{n}_{j k,i}-d_{j k}.$ Thus we have
\begin{equation*}
    \left\{ \|D_{n}-D\|_{\infty}\ge t \right\}=\bigcup\limits_{1\le j,k\le p}\{ |d^{n}_{j k}-d_{j k}|\ge t \}=\bigcup\limits_{1\le j,k\le p} \left\{ \left|\frac{1}{n}\sumn e^{n}_{j k,i}\right| \ge t \right\} .
\end{equation*}
{By the Cauchy-Schwarz inequality and $\mathrm{E}{Y^{2k}} \le \frac{{(2k)!}}{{{2^k}k!}}\left\| Y\right\|_{G}^{2k}$ for $k\ge 1$, we have
\begin{align*}
    \operatorname{E}(|d^{n}_{j k,i}|^{l})
    =\operatorname{E}(|X_{i,j}|^{l}|X_{i,k}|^{l})& \le (\operatorname{E}|X_{i ,j}|^{2l})^{\frac{1}{2}}(\operatorname{E}|X_{i ,k}|^{2l})^{\frac{1}{2}}=(\operatorname{E}|X_{j}|^{2l})^{\frac{1}{2}}(\operatorname{E}|X_{k}|^{2l})^{\frac{1}{2}}\le  \frac{{(2l)!}}{{{2^l}l!}}\left(\| X_{j}\|_{G}^{2l}\| X_{k}\|_{G}^{2l}\right)^{1/2}\\
[\text{By Assumption \ref{ass-1}a}]~   & \le \frac{{(2l)!}}{{{2^l}(l!)^2}} M_{1}^{2l}l!\le \frac{({2^l}l!)^2}{{{2^l}(l!)^2}} M_1^{2l}l!\\
&=2^lM_{1}^{2l}l!=\frac{8M_1^4}{2}(2M_1^2)^{l-2}\cdot l!,
\end{align*}
where we use the growth of sub-Gaussian moments condition of $\{X_{j}\}_{j=1}^p$ in the second last inequality, and the last inequality stems from:
\begin{align*}
(2k)!=(2k)(2k - 1)(2k - 2) \cdots 4 \cdot 3 \cdot 2 \cdot 1 &= {2^k} \cdot k!(2k - 1)(2k - 3) \cdots 3 \cdot 1\\
& \le {2^k} \cdot k!(2k)(2k - 2) \cdots 4 \cdot 2 = {({2^k}k!)^2}~\text{for}~k\ge1.
\end{align*}

Using $|a-b|^{l}\le 2^{l}(|a|^{l}+|b|^{l})$, we have
\begin{equation*}
    |e^{n}_{j k,i}|^{l}= |d^{n}_{j k,i}-\operatorname{E}d^{n}_{j k,i}|^{l} \le 2^{l}(|d^{n}_{j k,i}|^{l}+|\operatorname{E}d^{n}_{j k,i}|^{l}).
\end{equation*}
Take expectation and use the Jensen's inequality $|\operatorname{E}d^{n}_{j k,i}|^{l}\le \operatorname{E} |d^{n}_{j k,i}|^{l}  $, we have
\begin{equation}\label{eq:grow}
    \operatorname{E}(|e^{n}_{j k,i}|^{l})
    \le 2^{l+1}\operatorname{E} |d^{n}_{j k,i}|^{l}
    \le 2^{l+1} \frac{8M_1^4}{2}(2M_1^2)^{l-2}\cdot l!
    = \frac{64M_1^4}{2}(4M_1^2)^{l-2}\cdot l!.
\end{equation}
For the independent random variables $\{e^{n}_{j k,i}\}_{i=1}^{n}$, by the Bernstein's inequality with the growth of moments condition \eqref{eq:grow} (see Corollary 4.6 in \cite{zhang2020concentration}), we have
\begin{equation*}
    P\left(\left|\frac{1}{n}\sumn e^{n}_{j k,i}\right|\ge t\right)\le
    2\exp\left( -\frac{n^{2}t^{2}}{128M_1^4 n+8M_1^2n t} \right)
    =2\exp\left( -\frac{n t^{2}}{8(16 M_{1}^{4}+M_1^2 t)} \right).
\end{equation*}
Thus we conclude
\begin{equation*}
    P(\|D_{n}-D\|_{\infty}\ge t)\le \sum\limits_{1\le j,k\le p} P\left(\left|\frac{1}{n}\sumn e^{n}_{j k,i}\right|\ge t\right)
    \le 2p^{2}\exp\left( -\frac{n t^{2}}{8(16 M_{1}^{4}+M_1^2 t)} \right).
\end{equation*}
}
\end{proof}

\subsubsection{Lemma \ref{3-7}}
\hspace{4mm} The Lemma \ref{3-7} shows we can use $\frac{3}{2}\opnorm{D^{-1}}$ to bound $\opnorm{D_{n}^{-1}}$ from above.
\begin{lemma}\label{3-7}
Under the Assumption \ref{ass-1}, for any $\delta_{5}\in(0,1)$, {let
\begin{equation}\label{n1}
    N_{1}:=24p\opnorm{D^{-1}}(48p\opnorm{D^{-1}}M_1^4+M_1^2)\log\left(\frac{2p^{2}}{\delta_{5}}\right),
\end{equation}
we have when $n>N_{1}$
\begin{equation*}
    P\left(\opnorm{D_{n}^{-1}}\le \frac{3}{2}\opnorm{D^{-1}}\right)\ge 1-\delta_{5}.
\end{equation*}
}
\end{lemma}
\begin{proof}
Notice the fact that if $A,B\in\mathbb{R}^{p\times p}$ are invertible and $\opnorm{A^{-1}}\opnorm{A-B}<1$, then
\begin{equation*}
    \opnorm{A^{-1}-B^{-1}}\le \frac{\opnorm{A^{-1}}^{2}\opnorm{A-B}}{1-\opnorm{A^{-1}}\opnorm{A-B}}.
\end{equation*}
(see Lemma E.4 in \cite{Sun2017GraphicalNO}). Let $A=D$ and $B=D_{n}$, when $\opnorm{D^{-1}}\opnorm{D-D_{n}}\le\frac{1}{3}$, we have
\begin{align*}
    \opnorm{D^{-1}-D_{n}^{-1}}
    \le \frac{\frac{1}{3}\opnorm{D^{-1}}}{1-\frac{1}{3}}=\frac{1}{2}\opnorm{D^{-1}},
\end{align*}
from which we have $\opnorm{D_{n}^{-1}}\le \opnorm{D^{-1}}+\opnorm{D^{-1}-D_{n}^{-1}}\le \frac{3}{2}\opnorm{D^{-1}}.$ Therefore, it gives
\begin{equation*}
    P\left(\opnorm{D_{n}^{-1}}\le \frac{3}{2}\opnorm{D^{-1}}\right)\ge P\left(\opnorm{D^{-1}}\opnorm{D-D_{n}}\le\frac{1}{3}\right).
\end{equation*}
\par {Recall that we have $\opnorm{A}\le p\|A\|_{\infty}$ for $A\in \mathbb{R}^{p\times p}$, by the Lemma \ref{3-6} with $t=1/(3p\opnorm{D^{-1}})$, we have
\begin{align*}
    P\left(\opnorm{D_{n}^{-1}}\ge \frac{3}{2}\opnorm{D^{-1}}\right)
    &\le P\left(\opnorm{D^{-1}}\opnorm{D-D_{n}}\ge\frac{1}{3}\right) \le P\left(p\opnorm{D^{-1}}\| D-D_{n} \|_{\infty}\ge \frac{1}{3}\right)\\
    &\le 2p^{2}\exp \left( -\frac{ n }{ 24p\opnorm{D^{-1}}(48 p\opnorm{D^{-1}}M_1^4+M_1^2) } \right).
\end{align*}}
Let $N_1$ be defined in \eqref{n1}, as $n>N_{1}$, we get $P\left(\opnorm{D_{n}^{-1}}\ge \frac{3}{2}\opnorm{D^{-1}}\right)\le\delta_{5}$.
\end{proof}

\section{The auxiliary lemmas and results}\label{sec-8}
\hspace{4mm}The Lemma \ref{3-1}, Lemma \ref{3-4}, Inequalities \eqref{ineq-1} and \eqref{ineq-2} are from the proof of \cite{tong2018analysis}. We provide all the complete proofs in this section for integrity.

\subsection{Proof of Lemma \ref{3-1}}
\hspace{4mm}Define the random variable
\begin{equation*}
    \xi(f):=(T+\lambda_{n}I)^{-\frac{1}{2}}\langle \K Y,f \rangle\K Y
    \quad\text{and}\quad
    \xi_{i}(f):=(T+\lambda_{n}I)^{-\frac{1}{2}}\langle \K Y_{i},f \rangle\K Y_{i}\quad(1\le i\le n).
\end{equation*}
Then $\xi_{i}$ are independent copies of $\xi$, which takes values in $\mathrm{HS}(\mathcal{T})$.

Recall we define $\{(\tau_{k},\varphi_{k})\}_{k\ge1}$ to be the set of eigenvalue-eigenfunction pairs of the operator $T$,
\begin{align*}
    \|\xi\|^{2}_{\operatorname{HS}}
    &=\suminf\|(T+\lambda_{n}I)^{-\frac{1}{2}}\langle \K Y,\varphi_{k} \rangle\K Y\|^{2}\le \|(T+\lambda_{n}I)^{-\frac{1}{2}}\|_{\operatorname{op}}^{2}\|\K Y\|^{2}\suminf |\langle \K Y,\varphi_{k} \rangle|^{2}\\
    &= \|(T+\lambda_{n}I)^{-\frac{1}{2}}\|_{\operatorname{op}}^{2}\|\K Y\|^{4}\le \frac{\kappa^{4}M_{2}^{4}}{\lambda_{n}},
\end{align*}
where we use the fact
$\suminf |\langle \K Y,\varphi_{k} \rangle|^{2}=\|\K Y\|^{2}$, $\|(T+\lambda_{n}I)^{-\frac{1}{2}}\|_{\operatorname{op}}\le \frac{1}{\sqrt{\lambda_{n}}}~\text{and}~\|\K Y\|\le \|\K\|_{\operatorname{op}}\|Y\|\le\kappa M_{2}.$

 Notice $\K Y$ can be expended to $\sum\limits_{l=1}^{+\infty} \innerproduct{\K Y}{\varphi_{l}} \varphi_{l}$, we have
\begin{align*}
    &~~~~\|\xi\|_{\operatorname{HS}}^{2}
    =\suminf\left\|\sum\limits_{l=1}^{+\infty}\innerproduct{\K Y}{\varphi_{k}}\innerproduct{\K Y}{\varphi_{l}}(T+\lambda_{n}I)^{-\frac{1}{2}}\varphi_{l}\right\|^{2}=\left(\suminf|\innerproduct{\K Y}{\varphi_{k}}|^{2}
    \right)\left\|\sum\limits_{l=1}^{+\infty}\innerproduct{\K Y}{\varphi_{l}}(T+\lambda_{n}I)^{-\frac{1}{2}}\varphi_{l}\right\|^{2}\\
    &=\|\K Y\|^{2} \left\|\sum\limits_{l=1}^{+\infty} \frac{1}{\sqrt{\tau_{l}+\lambda_{n}}}\innerproduct{\K Y}{\varphi_{l}}\varphi_{l} \right\|^{2}=\|\K Y\|^{2} \sum\limits_{l=1}^{+\infty}\frac{1}{\tau_{l}+\lambda_{n}}|\innerproduct{\K Y}{\varphi_{l}}|^{2}
    \le \kappa^{2}M_{2}^{2} \sum\limits_{l=1}^{+\infty}\frac{1}{\tau_{l}+\lambda_{n}}|\innerproduct{\K Y}{\varphi_{l}}|^{2}.
\end{align*}
Using \eqref{2-7} we have
\begin{equation*}
    \operatorname{E}[|\innerproduct{\K Y}{\varphi_{l}}|^{2}]
    =\operatorname{E}[|\innerproduct{Y}{\K \varphi_{l}}|^{2}]
    =\innerproduct{\K \varphi_{l}}{L_{C}\K \varphi_{l}}
    =\innerproduct{T\varphi_{l}}{\varphi_{l}}=\tau_{l},
\end{equation*}
from which we get
\begin{equation*}
    \operatorname{E}(\| \xi \|_{\operatorname{HS}}^{2})\le \kappa^{2}M_{2}^{2}\sum\limits_{l=1}^{+\infty}\frac{\tau_{l}}{\tau_{l}+\lambda_{n}}=\kappa^{2}M_{2}^{2}D(\lambda_{n}).
\end{equation*}
Notice
\begin{align*}
    \operatorname{E}(\innerproduct{\K Y}{f}\innerproduct{\K Y}{g})
    &=\operatorname{E}(\innerproduct{Y}{\K f}\innerproduct{Y}{\K g})=\operatorname{E}\iint_{\mathcal{T}\times\mathcal{T}}Y(s)Y(t)(\K f)(s)(\K g)(t)\mathrm{d}s\mathrm{d}t\\
    &=\int_{\mathcal{T}}\left( \int_{\mathcal{T}} C(s,t)(\K f)(s) \mathrm{d}s\right)(\K g)(t)\mathrm{d}t=\innerproduct{L_{C}\K f}{\K g}=\innerproduct{T f}{g},
\end{align*}
from which we have
$
    \operatorname{E}(\innerproduct{\K Y}{f}\K Y)=T(f).
$
Therefore $(\operatorname{E}\xi)(f)=(T+\lambda_{n}I)^{-\frac{1}{2}}T(f).$
\par Taking $\|\xi\|_{\operatorname{H}}\le \frac{\kappa^{2}M_{2}^{2}}{\sqrt{\lambda_{n}}}$ and $\operatorname{E}(\|\xi\|_{\operatorname{H}}^{2})\le \kappa^{2}M_{2}^{2}D(\lambda_{n})$ in the Lemma \ref{3-3}, we have for any $\delta_{1}\in(0,2e^{-1})$, with probability at least $1-\delta_{1}$
\begin{align*}
    \|(T+\lambda_{n})^{-\frac{1}{2}}(T_{n}-T)\|_{\operatorname{op}}
    &\le \left\|\frac{1}{n}\sumn (\xi_{i}-\operatorname{E}\xi_{i})\right\|_{\operatorname{HS}}\\
    &\le \frac{2\kappa^{2}M_{2}^{2}\log\left(\frac{2}{\delta_{1}}\right)}{n\sqrt{\lambda_{n}}}+\sqrt{\frac{2\kappa^{2}M_{2}^{2}D(\lambda_{n})\log\left(\frac{2}{\delta_{1}}\right)}{n}}\le c_{1}\log\left(\frac{2}{\delta_{1}}\right)B_{n},
\end{align*}
where we let $c_{1}:=2\kappa^{2}M^{2}_{2}$ and $B_{n}:=\frac{1}{n\sqrt{\lambda_{n}}}+\sqrt{\frac{D(\lambda_{n})}{n}}$. $\hfill\square$

\subsection{Lemma \ref{3-3}, Lemma \ref{3-8} and Corollary \ref{Lip.Class.lem}}
\hspace{4mm}The Lemma \ref{3-3} and Lemma \ref{3-8} can be seen as the Bernstein-type concentration inequalities for the random variables taking values in a Hilbert space. In the Lemma \ref{3-3}, we assume the random variables to be bounded with regard to the norm in the Hilbert space, while we assume that the random variables satisfy the Bernstein's growth of moments condition in the Lemma \ref{3-8}. The first lemma is based on Theorem 3.3.4a in \cite{yurinsky2006sums}.
\begin{lemma}\label{3-8}
Let $\mathcal{H}$ be a Hilbert space endowed with norm $\|\cdot\|_{\operatorname{H}}$. Let $\{\xi_{i}\}_{i=1}^{n}$ be a sequence of $n$ independent random variables in $\mathcal{H}$ with zero mean. Assume there exist $B,M >0$ such that for all integers $l\ge 2$:
\begin{center}
$\operatorname{E}(\|\xi_i\|_{\operatorname{H}}^{l})\le \frac{B^{2}}{2}l!M^{l-2},~i=1,2,\cdots,n$,
\end{center}
then for any $\delta\in(0,1)$, we have
\begin{equation*}
    P\left( \left\|\frac{1}{n}\sumn\xi_{i}\right\|_{\operatorname{H}}\ge \frac{2M\log\left(\frac{2}{\delta}\right)}{n}+\sqrt{\frac{2\log\left(\frac{2}{\delta}\right)}{n}}B \right)\le \delta.
\end{equation*}
\end{lemma}
\begin{proof}
We refer readers to Lemma 2 in \cite{lv2012integral} for the proof of this lemma.
\end{proof}

The growth of moments condition \label{eq:Bernstein} with $l=2$ gives $B=\sqrt{\max\limits_{1\le i \le n}\operatorname{E}(\|\xi_i\|_{\operatorname{H}}^{2})}$. Then we have the following lemma.
\begin{lemma}\label{3-3}
Let $\mathcal{H}$ be a Hilbert space endowed with norm $\|\cdot\|_{\operatorname{H}}$ and $\{\xi_{i}\}_{i=1}^{n}$ be a sequence of $n$ independent random variables  taking values in $\mathcal{H}$. Assume that $\|\xi\|_{\operatorname{H}}\le M$ (a.s.), then for any $\delta\in(0,1)$, with probability at least $1-\delta$.
\begin{equation*}
    \left\|\frac{1}{n}\sumn (\xi_{i}-\operatorname{E}\xi_{i})\right\|_{\operatorname{H}}\le
    \frac{2M\log\left(\frac{2}{\delta}\right)}{n}+
    \sqrt{\frac{2\max\limits_{1\le i \le n}\operatorname{E}(\|\xi_i\|_{\operatorname{H}}^{2})\log\left(\frac{2}{\delta}\right)}{n}}
\end{equation*}
\end{lemma}
\begin{proof}
For $l>2$, one has $\operatorname{E}(\|\xi_i\|_{\operatorname{H}}^{l})\le M^{l-2}\operatorname{E}(\|\xi_i\|_{\operatorname{H}}^2) < \frac{B^{2}}{2}l!M^{l-2},~i=1,2,\cdots,n$ with $B=\sqrt{\max\limits_{1\le i \le n}\operatorname{E}(\|\xi_i\|_{\operatorname{H}}^{2})}$. The result follows from Lemma \ref{3-8}.
\end{proof}

Next, we consider the sub-Gaussian concentration for norm of sum of random vectors in the Hilbert space. A centered random variable $X$ is called sub-Gaussian if
\begin{center}
${\rm{E}}{e^{t X}} \le {e^{{t^2}{{\sigma}^2}/2}},~\forall~t \in \mathbb{R}$,
\end{center}
where the quantity ${\sigma^2}>0$ is named as the sub-Gaussian \emph{variance proxy} [see \cite{zhang2020concentration}, denote as $X \sim \operatorname{subG}(\sigma^{2})$]. Let ${\left\| X \right\|_{G}} =\sup_{k \ge 1} {[ {\frac{{\rm{E}}{X^{2k}}}{(2 k-1) ! !}} ]^{1/(2k)}}$ be the sub-Gaussian norm, and $X$ is sub-Gaussian if ${\left\| X \right\|_{G}}<\infty$; see \cite{buldygin2000metric}. In order to derive concentration for the norm of sum of sub-Gaussian vectors in Hilbert space, we consider following framework in \cite{maurer2021concentration}.

Let $Z=\left(Z_{1}, \ldots, Z_{n}\right)$ be a vector of independent data with values in a space ${\operatorname{H}},$ define $Z^{\prime}$ as an independent copy of $Z$. Given a function
$f: {\operatorname{H}}^{n} \rightarrow \mathbb{R},$ it is of interest to study the concentration inequality for $f(Z)-\mathrm{E}f(Z)$. A special case is $f(Z)$ the norm function of data. For $w \in {\operatorname{H}}$ and $k \in\{1, \ldots, n\}$ define the substitution operator $S_{w}^{k}: {\operatorname{H}}^{n} \rightarrow {\operatorname{H}}^{n}$ by
$$
S_{w}^{k}(z)=\left(z_{1}, \ldots, z_{k-1}, w, z_{k+1}, \ldots, z_{n}\right)
$$
and the centered conditional version of $f$ as the random variable is given by
\begin{align}\label{eq:identy}
D_{f, Z_k}(z) & \equiv f\left(z_{1}, \ldots, z_{k-1}, {Z_{k}}, z_{k+1}, \ldots, z_{n}\right)-\mathrm{E}[f(z_{1}, \ldots, z_{k-1}, {Z_{k}^{\prime}}, z_{k+1}, \ldots, z_{n})]\nonumber\\
&=f\left(S_{Z_{k}}^{k}(z)\right)-\mathrm{E}[f(S_{Z_{k}^{\prime}}^{k}(z))]=\mathrm{E}[f\left(S_{Z_{k}}^{k} (z)\right)-f(S_{Z_{k}^{\prime}}^{k}(z))|Z_k].
\end{align}

 Here $\{D_{f, Z_k}(z)\}_{k=1}^n$ can be viewed as random-valued functions $z \in {\operatorname{H}}^{n} \mapsto D_{f, Z_k}(z)$. If $f(z)=\sum_{i=1}^{n} z_{i}$ then $D_{f, Z_k}(Z)= Z_{k}-\mathrm{E}Z_{k}$ is independent of $z .$

Based on the norm ${\left\|\cdot\right\|_{G}}$, we aim to obtain a tighter and extended McDiarmid's inequality for $f(X)-\mathrm{E}f(X)$ with stochastic bounded difference conditions concerning the structure of $f$. The following corollary is a tighter sub-Gaussian concentration for the norm of sum of sub-Gaussian vectors in Hilbert space, comparing to Theorem 3 in \cite{maurer2021concentration}.
\begin{corollary}\label{cor:fxcon}
Suppose that $\{{{D_{f,{Z_i}}}(z)}\}_{i=1}^n$ have zero mean defined by \eqref{eq:identy}. If $\{D_{f,  Z_i}(z)\}_{i=1}^n$ have finite $\| \cdot \|_G$-norm for all ${z \in {\operatorname{H}}}$, we have, $\forall t \ge 0$
\begin{center}
$f(Z)- {\rm{E}}f(Z) \sim \operatorname{subG}(8\mathop {\sup }\limits_{z \in {\operatorname{H}}} \sum\limits_{i = 1}^n {\left\| {{D_{ f,{Z_i}}}(z)} \right\|_{G}^2} )$ and
$P \left\{ {f(Z) - {\rm{E}}f(Z) > t} \right\} \le \exp ( {\frac{{ - {t^2}}}{{16\mathop {\sup }\limits_{z \in {\operatorname{H}}} \sum\limits_{k = 1}^n {\left\| {{D_{f,{Z_k}}}(z)} \right\|_{G}^2} }}} ).$
\end{center}
\end{corollary}

\begin{proof}
Let the \emph{tilted expectation} $\mathrm{E}_{Y}$ be ${\mathrm{E}_Y}[Z] = \mathrm{E}\left[ {Z \cdot \frac{{{e^Y}}}{{\mathrm{E}{e^Y}}}} \right]$ with the exponential weighted ${\frac{{{e^Y}}}{{\mathrm{E}{e^Y}}}}$. The proof is based on the entropy  of a random variable $Y$ defined by
$$S(Y):=\mathrm{E}_{Y}[Y]-\log \mathrm{E}\left[e^{Y}\right]=\mathrm{E}\left[ {Y \cdot \frac{{{e^Y}}}{{\mathrm{E}{e^Y}}}} \right]-\log \mathrm{E}\left[e^{Y}\right],$$
which is free of centering, i.e. $S(Y-\mathrm{E}Y)=S(Y)$.

Suppose that $\mathrm{E}{Y}=0$, by Jensen's inequality we have
\begin{align}\label{eq:entropyin}
S(Y)=\mathrm{E}_{Y}\left[\log\left(\frac{e^{Y}}{\mathrm{E}e^{Y}}\right)\right]  \le \log \mathrm{E}_{Y}\left[\frac{e^{Y}}{\mathrm{E}e^{Y}}\right]=\log \mathrm{E}\left[\frac{e^{Y}}{\mathrm{E}e^{Y}}\cdot \frac{{{e^Y}}}{{\mathrm{E}{e^Y}}}\right]=\log \mathrm{E}e^{2 Y}-2 \log \mathrm{E}e^{Y} \le \log \mathrm{E}e^{2 Y},
\end{align}
where the last inequality is also derived by  Jensen's inequality ${\rm{E}}{e^Y} \ge {e^{{\rm{E}}Y}} = 1.$

The concentration inequality is by Cramer-Chernoff method with the logarithm of the
MGF represented as the integral of entropy (Theorems 1 in \cite{maurer2012thermodynamics}).
\begin{align}\label{eq:logmgf}
\log \mathrm{E}\left[e^{t(Y-\mathrm{E}Y)}\right] = \log \mathrm{E}\left[e^{t Y}\right] = t \int_{0}^{t} \frac{S(\gamma Y) d \gamma}{\gamma^{2}},~\forall~t >0
\end{align}
and the \emph{subadditivity of entropy} (see Theorems 6 and Section 3.1 in \cite{maurer2012thermodynamics})
\begin{align}\label{eq:subadditivity}
S(f(Z)) \le {{\rm{E}}_{f(Z)}}[ {\sum\limits_{i = 1}^n S ({D_{f,{Z_i}}}(Z))}],~\text{where we denote}~S({D_{f,{Z_i}}}(Z)):=S({D_{f,{Z_k}}}(z))|_{z=Z}.
\end{align}

If $Y$  has zero mean, then ${\rm{E}}{e^{tY}} = 1 + \sum\limits_{k = 2}^\infty  {\frac{{{t^k}{\rm{E}}{Y^k}}}{{k!}}}$. Note that by Cauchy's inequality and arithmetic-geometric mean inequality
 \[{\rm{E}}|tY|^{2k + 1} \le {\left( {{\rm{E}}|tY|^{2k}{\rm{E}}|t Y|^{2k + 2}} \right)^{1/2}} \le \frac{1}{2}\left( {{t^{2k}}{\rm{E}}{Y^{2k}} + {t^{2k + 2}}{\rm{E}}{Y^{2k + 2}}} \right).\]
Then, $\frac{{{\rm{E}}|tY{|^3}}}{{3!}} \le \frac{1}{{2 \cdot 3!}}\left( {{t^2}{\rm{E}}{Y^2} + {t^4}{\rm{E}}{Y^4}} \right)$ implies
\begin{align}\label{eq:oddmoments}
{\rm{E}}{e^{tY}}& \le 1 + \left( {\frac{1}{2} + \frac{1}{{2 \cdot 3!}}} \right){t^2}{\rm{E}}{Y^2} + \sum\limits_{k = 2}^\infty  {\left( {\frac{1}{{(2k)!}} + \frac{1}{2}\left[ {\frac{1}{{(2k - 1)!}} + \frac{1}{{(2k + 1)!}}} \right]} \right)} {t^{2k}}{\rm{E}}{Y^{2k}}\nonumber\\
&\le \sum\limits_{k = 0}^\infty  {{2^k}} \frac{{{t^{2k}}{\rm{E}}{Y^{2k}}}}{{(2k)!}} \le \exp \{{{t^2{{\left\| Y\right\|}_{G}^2}}}\},
\end{align}
where the last inequality is by the definition of $\left\| Y \right\|_{G}<\infty$ and $\mathrm{E}{Y^{2k}} \le \frac{{(2k)!}}{{{2^k}k!}}\left\| Y\right\|_{G}^{2k}$.

Thus \eqref{eq:oddmoments} shows
$\log {\rm{E}}{e^{2tY}}\le {4{t^2}\left\| Y\right\|_{G}^2}$ and \eqref{eq:entropyin} gives
\begin{align}\label{eq:syG}
S({D_{f,{Z_i}}}(z))  \le \log \mathrm{E}e^{2 {D_{f,{Z_i}}}(z) } \le {4 \left\| {D_{f,{Z_i}}}(z)\right\|_{G}^2},~{z \in {\operatorname{H}}}.
\end{align}

Denote $\left\| {D_{f,{Z_i}}}(Z)\right\|_{G}^2:=\left\| {D_{f,{Z_i}}}(z)\right\|_{G}^2|_{z=Z}$. Combing \eqref{eq:logmgf} and \eqref{eq:subadditivity}, we obtain
\begin{align*}
\log \mathrm{E}\left[e^{t(f(Z)-\mathrm{E} f(Z))}\right]&=t \int_{0}^{t} \frac{S(\gamma f(Z)) d \gamma}{\gamma^{2}} \leq t \int_{0}^{t} \frac{1}{\gamma^{2}} \mathrm{E}_{\gamma f(Z)}\left[\sum_{i=1}^{n} S\left(D_{ \gamma f, Z_{i}}(Z)\right)\right] d \gamma \\
& = t \int_{0}^{t} \frac{1}{\gamma^{2}} \mathrm{E}_{\gamma f(Z)}\left[\sum_{i=1}^{n} S\left(\gamma D_{f, Z_{i}}(z)\right)|_{z=Z}\right] d \gamma \\
[\text{By}~\eqref{eq:syG}]&\le t \int_{0}^{t} \frac{4 \gamma^{2}}{\gamma^{2}} \mathrm{E}_{\gamma f(Z)}\left[\sum_{i=1}^{n}\left\|D_{ f, Z_{i}}(z)\right\|_{G}^{2}|_{z=Z}\right] d \gamma\\
& \le t \int_0^t  {\frac{{4{\gamma ^2}}}{{{\gamma ^2}}}} {{\rm{E}}_{\gamma f(Z)}}\left[ {\mathop {\sup }\limits_{z \in {\operatorname{H}}} \sum\limits_{i = 1}^n {\left\| {{D_{ f,{Z_i}}}(z)} \right\|_{G}^2} } \right]d\gamma=8\mathop {\sup }\limits_{z \in Z} \sum\limits_{i = 1}^n {\left\| {{D_{f,{Z_i}}}(z)} \right\|_{G}^2}  \cdot \frac{{{t^2}}}{{\rm{2}}},
\end{align*}
which shows
$f(Z) - \operatorname{E}f(Z) \sim \operatorname{subG}\left(8\mathop {\sup }\limits_{z \in {\operatorname{H}}} \sum\limits_{i = 1}^n {\left\| {{D_{ f,{Z_i}}}(z)} \right\|_{G}^2}\right).$
\end{proof}

\begin{corollary}[Hoeffding-type inequality for norm of vector]\label{Lip.Class.lem}
Suppose $\{\xi_i\}_{i=1}^n$ are independent random elements with values in a Hilbert space ${\operatorname{H}}$ s.t. ${\max }_{i \in [n]}\|\|\xi_{i}\|_{\operatorname{H}}\|_{G} < \infty$.  Then, with probability at least $1-\delta$
\begin{align}\label{eq:subg}
\left\|\frac{1}{n}\sumn (\xi_{i}-\operatorname{E}\xi_{i})\right\|_{\operatorname{H}} \leq \frac{1}{\sqrt n}\left\{\sqrt{\frac{1}{n} \sum\limits_{i=1}^n{\E \| \xi_{i}-\operatorname{E}\xi_{i}\|^2}}+8\sqrt{\frac{1}{n}\sum\limits_{i=1}^n \| \| \xi_{i}\|_{\operatorname{H}}\|_{G}^2 \log (\frac{1}{\delta})}\right\}.
\end{align}
\end{corollary}
\begin{proof}
    Consider the function $f(\boldsymbol{z}) := \big\| \sum_{i=1}^n \boldsymbol{z}_i - \boldsymbol{v} \big\|$ for any $ \boldsymbol{v},\boldsymbol{z}_i \in {\operatorname{H}}$. From \eqref{eq:identy}, we have
    \begin{align}\label{eq:1LC}
 \mathrm{E}[f\big( S_{\xi_k}^k (\boldsymbol{z})\big) - f\big( S_{\xi_k^{\prime}}^k (\boldsymbol{z})\big)|\xi_k ]&= \E \Big[\bigg| \Big\| \sum_{i \neq k} \boldsymbol{z}_i + \xi_k - \boldsymbol{v}\Big\| -  \Big\| \sum_{i \neq k} \boldsymbol{z}_i + \xi_k^{\prime} - \boldsymbol{v}\Big\| \bigg| \, \Big| \, \xi_k \Big] \leq \E \big[ \| \xi_k - \xi_k^{\prime}\|_{\operatorname{H}} \, \big| \, \xi_k \big],
    \end{align}
 for $k=1,2,\cdots,n$, which directly implies $\| D_{f, \xi_k} (\boldsymbol{z}) \|_G \le \| \| \xi_k-\xi_k^{\prime} \|_{\operatorname{H}} \|_G $ by following derivation:
\begin{align}\label{eq:identy2}
\|D_{f, \xi_k}(\boldsymbol{z}) \|_G &=  \|f\left(z_{1}, \ldots, z_{k-1}, \xi_{k}, z_{k+1}, \ldots, x_{n}\right) - \mathrm{E}[f(z_{1}, \ldots, z_{k-1}, \xi_{k}^{\prime}, z_{k+1}, \ldots, z_{n})] \|_G\nonumber\\
&=\|\mathrm{E}[f(S_{\xi_{k}}^{k} (\boldsymbol{z}))-f(S_{\xi_{k}^{\prime}}^{k} (\boldsymbol{z}))|\xi_k] \|_G\nonumber\\
[\text{By~}\eqref{eq:1LC}]~& \le \|\mathrm{E}[\| \xi_k - \xi_k^{\prime}\|_{\operatorname{H}}|\xi_k]\|_G.
\end{align}
The conditional Jensen's inequality and Holder's inequality show for any $k \geq 1$
\begin{align*}
\mathrm{E}\left[\left|\mathrm{E}\left[\| \xi_k - \xi_k^{\prime}\|_{\operatorname{H}}\mid {\xi_k}\right]\right|^{2k}\right]&\le \mathrm{E}\left[\mathrm{E}[\| \xi_k - \xi_k^{\prime}\|_{\operatorname{H}} \mid {\xi_k}]^{2k}\right]  \le \mathrm{E}\left\{\mathrm{E}\left[\| \xi_k - \xi_k^{\prime}\|_{\operatorname{H}}^{2k} \mid {\xi_k}\right]\right\}= \mathrm{E}\left|\| \xi_k - \xi_k^{\prime}\|_{\operatorname{H}}\right|^{2k}.
\end{align*}
Then the definition ${\left\| Y\right\|_{G}} = \sup_{k \ge 1} {\left[ {\frac{{{2^k}k!}}{{(2k)!}}{\rm{E}}{Y^{2k}}} \right]^{1/(2k)}}$ and \eqref{eq:identy2} show
\begin{align}\label{eq:identy3}
\| D_{f, \xi_k} (\boldsymbol{z}) \|_G \le \| \| \xi_k-\xi_k^{\prime} \|_{\operatorname{H}} \|_G \le 2\| \| \xi_k \|_{\operatorname{H}} \|_G.
\end{align}

Now we take $\boldsymbol{v}=\sum_{i=1}^n\E \xi_{i}$ and apply Corollary \ref{cor:fxcon} with the first inequality in \eqref{eq:identy3}. Let $e^{t}=\delta$ for solving $t$, we have with probability at least $1-\delta$, for $Z:=\left(\xi_{1}, \ldots, \xi_{n}\right)$
\begin{align*}
f(Z)-\E f(Z)=\left\|\sumn (\xi_{i}-\operatorname{E}\xi_{i})\right\|_{\operatorname{H}} - \E \left\|\sumn (\xi_{i}-\operatorname{E}\xi_{i})\right\|_{\operatorname{H}}& \le 4\sqrt{\sum_{k=1}^n \| \| \xi_k - \xi_k^{\prime} \|_{\operatorname{H}}\|_G^2 \log (1 / \delta)}\\
[\text{By~the second inequality in }\eqref{eq:identy3}]~ & \le 8\sqrt{\sum_{k=1}^n \| \|\xi_k \|_{\operatorname{H}}\|_G^2 \log (1 / \delta)}.
\end{align*}
Since ${\operatorname{H}}$ is a Hilbert space, $\{\xi_i\}_{i=1}^n$ are independent, and Jensen's inequality implies
\begin{center}
$\operatorname{E}\left\|\sumn (\xi_{i}-\operatorname{E}\xi_{i})\right\|_{\operatorname{H}}\leq \sqrt{\E\left\|\sumn (\xi_{i}-\operatorname{E}\xi_{i})\right\|_{\operatorname{H}}^2} = \sqrt{\sum_{i=1}^n \E \| \xi_i - \E \xi_i \|^2}.$
\end{center}
Then, with probability at least $1-\delta$, we have \eqref{eq:subg}.
\end{proof}
\subsection{Lemma \ref{3-4}}
\hspace{4mm}The Lemma \ref{3-4} shows the concentration property of $(T+\lambda_{n}I)^{-\frac{1}{2}}g_{n}$.
\begin{lemma}\label{3-4}
Under the Assumption \ref{ass-3}, for any $\delta_{3}\in(0,1)$, with probability at least $1-\delta_{3}$, there exists
\begin{equation*}
    \|(T+\lambda_{n}I)^{-\frac{1}{2}}g_{n}\|\le \frac{\sigma}{\sqrt{\delta_{3}}}B_{n}.
\end{equation*}
\end{lemma}
\begin{proof}
Define random variables $\xi$ and $\{\xi_{i}\}_{i=1}^n)$ taking values in the Hilbert space $L^{2}(\mathcal{T})$ by
\begin{equation*}
    \xi:=\varepsilon(T+\lambda_{n}I)^{-\frac{1}{2}}\K Y
    \quad\text{and}\quad
    \xi_{i}:=\varepsilon_{i}(T+\lambda_{n}I)^{-\frac{1}{2}}\K Y_{i},
\end{equation*}
where $\{\xi_{i}\}_{i=1}^{n}$ are independent copies of $\xi$ and $(T+\lambda_{n}I)^{-\frac{1}{2}}g_{n}=\frac{1}{n}\sumn\xi_{i}$.

{Noticing the random noise $\varepsilon$ has conditional zero mean given $Y$, we have
\begin{equation*}
    \operatorname{E}\xi=\operatorname{E}((T+\lambda_{n}I)^{-\frac{1}{2}}\K Y\cdot\operatorname{E}(\varepsilon|Y))=0.
\end{equation*}}

Expanding $\xi$ by the basis $\{\varphi_{k}:k\ge 1\}$ of the operator $T$, and we have
\begin{align*}
    \operatorname{E}(\|\xi\|^{2})
    &=\operatorname{E}\left( \left\| \suminf \innerproduct{\varepsilon(T+\lambda_{n}I)^{-\frac{1}{2}}\K Y}{\varphi_{k}}\varphi_{k} \right\|^{2} \right)=\operatorname{E}\left( \varepsilon^{2} \left\| \suminf \innerproduct{\K Y}{(T+\lambda_{n}I)^{-\frac{1}{2}}\varphi_{k}}\varphi_{k} \right\|^{2} \right)\\
    &=\operatorname{E}\left(\varepsilon^{2} \left\|\suminf \frac{1}{\sqrt{\tau_{k}+\lambda_{n}}}\innerproduct{\K Y}{\varphi_{k}}\varphi_{k}\right\|^{2} \right)=\operatorname{E}\left( \varepsilon^{2}\suminf\frac{|\innerproduct{\K Y}{\varphi_{k}}|^{2}}{\tau_{k}+\lambda_{n}} \right)\\
    &= \operatorname{E}\left( \suminf\frac{|\innerproduct{\K Y}{\varphi_{k}}|^{2}}{\tau_{k}+\lambda_{n}}\cdot \operatorname{E(\varepsilon^2| Y)} \right)    \le \sigma^{2}\suminf\frac{\operatorname{E}(|\innerproduct{\K Y}{\varphi_{k}}|^{2})}{\tau_{k}+\lambda_{n}} =\sigma^{2}\suminf\frac{\innerproduct{T\varphi_{k}}{\varphi_{k}}}{\tau_{k}+\lambda_{n}}=\sigma^{2}D(\lambda_{n}).
\end{align*}
Hence we have
\begin{equation*}
    \operatorname{E}(\|(T+\lambda_{n}I)^{-\frac{1}{2}}g_{n}\|^{2})
    =\operatorname{E}(\|\frac{1}{n}\sumn\xi_{i}\|^{2})=\frac{\operatorname{E}(\|\xi\|^{2})}{n} \le \frac{\sigma^{2}D(\lambda_{n})}{n}.
\end{equation*}
Using Markov inequality, we get $P(\|(T+\lambda_{n}I)^{-\frac{1}{2}}g_{n}\|\ge t)\le \frac{\sigma^{2}D(\lambda_{n})}{n t^{2}}$ from which we conclude with probability at least $1-\delta_{3}$, we have
\begin{equation*}
    \|(T+\lambda_{n}I)^{-\frac{1}{2}}g_{n}\|\le \frac{\sigma}{\sqrt{\delta_{3}}}\sqrt{\frac{D(\lambda_{n})}{n}}
    \le \frac{\sigma}{\sqrt{\delta_{3}}}B_{n}.
\end{equation*}
\end{proof}

\subsection{Two crucial inequalities}
\hspace{4mm}The following two inequalities play an important role in the proof of the Theorem \ref{thm-1}, which shows $\|(T+\lambda_{n}I)(T_{n}+\lambda_{n}I)^{-1}\|_{\operatorname{op}}$ and $\|(T+\lambda_{n}I)^{\frac{1}{2}}(T_{n}+\lambda_{n}I)^{-\frac{1}{2}}\|_{\operatorname{op}}$ can be bounded by $\opnorm{(T+\lambda_{n}I)^{-\frac{1}{2}}(T_{n}-T)}$.
\begin{inequality}\label{ineq-1}
$\opnorm{(T+\lambda_{n}I)(T_{n}+\lambda_{n}T)^{-1}}\le \left(\frac{1}{\sqrt{\lambda_{n}}}\opnorm{(T+\lambda_{n}I)^{-\frac{1}{2}}(T_{n}-T)}+1 \right)^{2}$.
\end{inequality}
\begin{proof}
Using the following decomposition of the operator product
\begin{equation*}
    BA^{-1}=(B-A)B^{-1}(B-A)A^{-1}+(B-A)B^{-1}+I
\end{equation*}
with $A=T_{n}+\lambda_{n}I$ and $B=T+\lambda_{n}I$, we have
\begin{align*}
    (T+\lambda_{n}I)(T_{n}+\lambda_{n}T)^{-1}
    &=(T-T_{n})(T+\lambda_{n}I)^{-1}(T-T_{n})(T_{n}+\lambda_{n}I)^{-1} +(T-T_{n})(T+\lambda_{n}I)^{-1}+I \\
    &=:F_{1}+F_{2}+I.
\end{align*}
\par For the operator $F_{1}$, we have
\begin{align*}
    \opnorm{F_{1}}
    &\le \opnorm{(T-T_{n})(T+\lambda_{n}I)^{-\frac{1}{2}}}
    \opnorm{(T+\lambda_{n}I)^{-\frac{1}{2}}(T-T_{n})}\cdot\frac{1}{\lambda_{n}}=\frac{1}{\lambda_{n}}\opnorm{(T+\lambda_{n}I)^{-\frac{1}{2}}(T_{n}-T)}^{2},
\end{align*}
where we use the fact $\opnorm{AB}=\opnorm{(AB)^{*}}=\opnorm{B^{*}A^{*}}=\opnorm{BA}$ for any self-adjoint operators $A$ and $B$, and the bound $\opnorm{(T_{n}+\lambda_{n}I)^{-1}}\le \frac{1}{\lambda_{n}}$.
\par For the operator $F_{2}$, applying $\opnorm{(T+\lambda_{n}I)^{-\frac{1}{2}}}\le\frac{1}{\sqrt{\lambda_{n}}}$, we have
\begin{align*}
    \opnorm{F_{2}}
    &\le \opnorm{(T-T_{n})(T+\lambda_{n}I)^{-\frac{1}{2}}}\opnorm{(T+\lambda_{n}I)^{-\frac{1}{2}}}\le \frac{1}{\sqrt{\lambda_{n}}}\opnorm{(T+\lambda_{n}I)^{-\frac{1}{2}}(T_{n}-T)}.
\end{align*}
Thus we obtain
\begin{align*}
    \opnorm{(T+\lambda_{n}I)(T_{n}+\lambda_{n}T)^{-1}}
    &\le \frac{1}{\lambda_{n}}\opnorm{(T+\lambda_{n}I)^{-\frac{1}{2}}(T_{n}-T)}^{2}+\frac{1}{\sqrt{\lambda_{n}}}\opnorm{(T+\lambda_{n}I)^{-\frac{1}{2}}(T_{n}-T)}+1 \\
    &\le \left(\frac{1}{\sqrt{\lambda_{n}}}\opnorm{(T+\lambda_{n}I)^{-\frac{1}{2}}(T_{n}-T)}+1 \right)^{2}.
\end{align*}
\end{proof}

\begin{inequality}\label{ineq-2}
$\|(T_{n}+\lambda_{n}I)^{-\frac{1}{2}}(T+\lambda_{n}I)^{\frac{1}{2}}\|_{\operatorname{op}}
    =\|(T+\lambda_{n}I)^{\frac{1}{2}}(T_{n}+\lambda_{n}I)^{-\frac{1}{2}}\|_{\operatorname{op}}\le\frac{1}{\sqrt{\lambda_{n}}}\opnorm{(T+\lambda_{n}I)^{-\frac{1}{2}}(T_{n}-T)}+1$.
\end{inequality}
\begin{proof}
Applying the fact  that
\begin{equation*}
\opnorm{A^{\gamma}B^{\gamma}}\le\opnorm{AB}^{\gamma},\quad \gamma\in(0,1)
\end{equation*}
for positive operators $A$ and $B$ defined on Hilbert space (see Lemma A.7 in \cite{blanchard2010optimal}), we have
\begin{align*}
    \|(T_{n}+\lambda_{n}I)^{-\frac{1}{2}}(T+\lambda_{n}I)^{\frac{1}{2}}\|_{\operatorname{op}}
    &=\|(T+\lambda_{n}I)^{\frac{1}{2}}(T_{n}+\lambda_{n}I)^{-\frac{1}{2}}\|_{\operatorname{op}} \\
    &\le \opnorm{(T+\lambda_{n}I)(T_{n}+\lambda_{n}T)^{-1}}^{\frac{1}{2}} \le \frac{1}{\sqrt{\lambda_{n}}}\opnorm{(T+\lambda_{n}I)^{-\frac{1}{2}}(T_{n}-T)}+1,
\end{align*}
where in the last step we use the inequality \eqref{ineq-1}.
\end{proof}

\subsection{Lemma \ref{KLlemma}}
\hspace{4mm} The Lemma \ref{KLlemma} is helpful in constructing the lower bound which is based on the testing multiple hypothesis.
\begin{lemma}\label{KLlemma}
Assume $N \ge 2$ and suppose there exists $\Theta=\{\theta_{i}\}_{i=0}^{N}$ such that the conditions are satisfied:
\begin{enumerate}
\item $2r$-separated condition: $d(\theta_{j}, \theta_{k}) \geq 2 r>0, \quad \forall\, 0 \leq j<k \leq N$,
\item  Kullback-Leibler average condition: if $P_{j} \ll P_{0}$ for $1\le j\le N$ and
\begin{center}
$\frac{1}{N} \sum_{j=1}^{N} K\left(P_{j} | P_{0}\right) \leq \rho \log N$ for some $0<\rho <\frac{1}{8}$ and $P_{j}=P_{\theta_{j}}\,(0\le j\le N)$.
\end{center}
\end{enumerate}
Then for all possible random variables $\tilde{\theta}$, we have
$$\inf _{\tilde{\theta}} \sup _{\theta \in \Theta} P_{\theta}(d(\tilde{\theta}, \theta) \geq  r) \geq \frac{\sqrt{N}}{1+\sqrt{N}}\left(1-2 \rho-\sqrt{\frac{2 \rho}{\log N}}\right)>0.$$
\end{lemma}
\begin{proof}
We refer readers to Theorem 2.5 in \cite{tsybakov2008} for the proof of this lemma.
\end{proof}

\subsection{Varhsamov-Gilbert Lemma}
\hspace{4mm} We need the Varhsamov-Gilbert lemma in the proof of the Theorem \ref{thm-2Euclidean} to construct the analogy of $\Theta$ in Lemma \ref{KLlemma}.
\begin{lemma}\label{VGlemma}
 Let $H\left(\theta, \theta^{\prime} \right)=\sum_{k=1}^{M} 1(\theta_{k} \neq \theta_{k}^{\prime})$ be the Hamming distance between elements $\theta, \theta^{\prime}$ in $\{0, 1\}^{M}$. For any integer $M \ge 8$, there exist vectors $\{ \theta^{i}\}_{i=0}^{N}\subset \{0,1\}^{M}$ such that
\begin{center}
\rm{(i)} $\theta^{0}=(0,\cdots,0)$,~~
\rm{(ii)} $H(\theta^{i}, \theta^{j}) > \frac{M}{8}$ for all $i \neq j$,~~
\rm{(iii)} $N \geq 2^{\frac{M}{8}}$.
\end{center}
\end{lemma}
\begin{proof}
We refer readers to page 104 in \cite{tsybakov2008} for the proof of this lemma.
\end{proof}

\subsection{Derivation of the equality \eqref{Ksigma2}}\label{App-C}
\hspace{4mm} Let $f_{1}(\bm{X},Y)$ be the density of $(\bm{X},Y)$ and $f_{2}(\varepsilon)=\frac{1}{\sigma \sqrt{2\pi}}e^{-\frac{\varepsilon^{2}}{2\sigma^{2}}}$ be the density of $\varepsilon$ since we assume $\varepsilon\sim N(0,\sigma^{2})$, then the density of $P_{\bm{\alpha}_1,\beta_{1}}$ can be written as
\begin{equation*}
    \frac{\mathrm{d}P_{\bm{\alpha}_1,\beta_{1}}}{\mathrm{d}\mu}(Z,\bm{X},Y)=f_{1}(\bm{X},Y)f_{2}(Z-\bm{X}^{T}\bm{\alpha}_1-\innerproduct{Y}{\beta_{1}}),
\end{equation*}
where $\mu$ is a dominant measure on the space $\mathbb{R}\times\mathbb{R}^{p}\times L^{2}(\mathcal{T})$.
\par {Therefore, under the assumption of $f_{2}(\varepsilon)=\frac{1}{\sigma \sqrt{2\pi}}e^{-\frac{\varepsilon^{2}}{2\sigma^{2}}}$, we have
\begin{align*}
    \log\left(\frac{\mathrm{d}P_{{\bm\alpha}_1,\beta_{1}}}{\mathrm{d}P_{{\bm\alpha}_2,\beta_{2}}}\right)
    &=\log\left( \frac{f_{2}(Z-\bm{X}^{T}\bm{\alpha}_1-\innerproduct{Y}{\beta_{1}})}{f_{2}(Z-\bm{X}^{T}\bm{\alpha}_2-\innerproduct{Y}{\beta_{2}})} \right) \\
    &=\frac{1}{2\sigma^{2}}\left((Z-\bm{X}^{T}\bm{\alpha}_2-\innerproduct{Y}{\beta_{2}})^{2}-(Z-\bm{X}^{T}\bm{\alpha}_1-\innerproduct{Y}{\beta_{1}})^{2} \right) \\
    &=\frac{1}{\sigma^{2}}(Z-\bm{X}^{T}\bm{\alpha}_1-\innerproduct{Y}{\beta_{1}})(\bm{X}^T (\bm{\alpha}_1-\bm{\alpha}_2)+\innerproduct{Y}{\beta_{1}-\beta_{2}})+ \frac{1}{2\sigma^{2}}(\bm{X}^T (\bm{\alpha}_1-\bm{\alpha}_2)+\innerproduct{Y}{\beta_{1}-\beta_{2}})^{2}.
\end{align*}}

{Notice when the true parameter is $({\bm\alpha}_1,\beta_{1})$, we have $Z=\bm{X}^{T}\bm{\alpha}_1+\innerproduct{Y}{\beta_{1}}+\varepsilon$, based on which we obtain
\begin{align*}
    &\operatorname{E}_{\bm{\alpha}_1,\beta_{1}}(Z-\bm{X}^{T}\bm{\alpha}_1-\innerproduct{Y}{\beta_{1}})(\bm{X}^T (\bm{\alpha}_1-\bm{\alpha}_2)+\innerproduct{Y}{\beta_{1}-\beta_{2}})\\
    &=\operatorname{E}\varepsilon(\bm{X}^T (\bm{\alpha}_1-\bm{\alpha}_2)+\innerproduct{Y}{\beta_{1}-\beta_{2}})=\operatorname{E}[(\bm{X}^T (\bm{\alpha}_1-\bm{\alpha}_2)+\innerproduct{Y}{\beta_{1}-\beta_{2}})\cdot \operatorname{E}(\varepsilon| \bm{X},Y)]=0.
\end{align*}}
{Thus the Kullback-Leibler distance $$K(P_{\bm{\alpha}_1,\beta_{1}} | P_{\bm{\alpha}_2,\beta_{2}})=\operatorname{E}_{\bm{\alpha}_1,\beta_{1}}\left(\log\left(\frac{\mathrm{d}P_{\bm{\alpha}_1,\beta_{1}}}{\mathrm{d}P_{\bm{\alpha}_2,\beta_{2}}}\right)\right)= \frac{1}{2\sigma^{2}} \operatorname{E}(\bm{X}^T(\bm{\alpha}_1-\bm{\alpha}_2)+\innerproduct{Y}{\beta_{1}-\beta_{2}})^{2}.$$}

$\hfill\square$

\subsection{Minimax rate with given Euclidean predictors}\label{se:FIX}

\begin{ass}\label{ass-60}
 For a fixed $\bm{\alpha}_{0}^{*}\in\mathbb{R}^{p}$ and different $\beta_{1},\beta_{2}\in\mathcal{H}(K)$. We assume that the Kullback-Leibler distance between $P_{\bm{\alpha}_{0}^{*},\beta_{1}}$ and $P_{\bm{\alpha}_{0}^{*},\beta_{2}}$ can be bounded by
$$
K(P_{\bm{\alpha}_{0}^{*},\beta_{1}} | P_{\bm{\alpha}_{0}^{*},\beta_{2}}):=\operatorname{E}_{\bm{\alpha}_{0}^{*},\beta_{1}}\log\left(\frac{\mathrm{d}P_{\bm{\alpha}_{0}^{*},\beta_{1}}}{\mathrm{d}P_{\bm{\alpha}_{0}^{*},\beta_{2}}}\right)\le K_{\sigma^{2}}\operatorname{E}(\innerproduct{Y}{\beta_{1}-\beta_{2}})^{2},
$$
where $K_{\sigma^{2}}>0$ is a variance-dependent constant and $\operatorname{E}_{\bm{\alpha}_{0}^{*},\beta_{1}}$ is the expectation taken over $P_{\bm{\alpha}_{0}^{*},\beta_{1}}$.
\end{ass}
\begin{corollary}[Minimax rate with given Euclidean predictors]\label{thm-2}
Under the Assumptions \ref{ass-3} and \ref{ass-60}, suppose the eigenvalues $\{\tau_{k}:k\ge 1 \}$ of the operator $T$ decay as $\tau_{k}=t_{0}k^{-2r}$ for some $r, t_0 \in(0,\infty)$, then for $\rho\in(0,\frac{1}{8})$, there exists a sequence $\{N_{n}\}_{n\ge1}$ satisfying
\begin{equation}
    \log\left(N_{n}\right)\ge \left(\frac{8}{\log2}\right)^{-\frac{2r}{1+2r}}(\frac{t_0K_{\sigma^{2}}}{\rho})^{\frac{1}{1+2r}}n^{\frac{1}{1+2r}},
\end{equation}
such that when $n\ge\frac{\rho\log2}{8t_0K_{\sigma^{2}}}$, the excess prediction risk satisfies
\begin{equation*}
    \inf_{\tilde{\eta}}\sup_{\eta_{0}\in \mathbb{R}^{p}\times\mathcal{H}(K)}
    P\left(\mathcal{E}(\tilde{\eta})-\mathcal{E}(\eta_{0})\ge \frac{t_0}{2^{4(1+r)}}(\frac{8t_0K_{\sigma^{2}}}{\rho\log2})^{-\frac{2r}{1+2r}}n^{-\frac{2r}{1+2r}}\right)\ge \frac{\sqrt{N_{n}}}{1+\sqrt{N_{n}}}\left(1-2 \rho-\sqrt{\frac{2 \rho}{\log N_{n}}}\right),
\end{equation*}
where we identify the prediction rule $\tilde{\eta}$ as the arbitrary estimator $(\tilde{\bm{\alpha}},\tilde{\beta})$ based on the training samples $\{ (Z_{i},\bm{X}_{i},Y_{i}) \}_{i=1}^{n}$, and view $\eta_{0}$ as the true parameter $(\bm{\alpha}_{0},\beta_{0})\in\mathbb{R}^{p}\times\mathcal{H}(K)$. We emphasize the probability $P$ is taken over the product space of training samples $\{ (Z_{i},\bm{X}_{i},Y_{i}) \}_{i=1}^{n}$ generated by $\eta_{0}=(\bm{\alpha}_{0},\beta_{0})$.
\end{corollary}
In \cite{cai2012minimax}, a asymptotically minimax lower bound is derived for the FLM in the Theorem 1, which is also a corollary of our result on the non-asymptotic and constant-specified minimax lower bound when $N_n \to \infty$.
\begin{proof}
Let $M$ be the smallest integer greater than $b_{0}n^{\frac{1}{1+2r}}$, where $b_{0}$ will be defined in later proof. For a binary sequence $\theta=(\theta_{M+1},\cdots,\theta_{2M})\in \{0,1 \}^{M}$, define
\begin{equation*}
\beta_{\theta}=M^{-\frac{1}{2}}\sum\limits_{k=M+1}^{2M}\theta_{k}\K \varphi_{k}.
\end{equation*}
\par By applying $\innerproduct{\K\varphi_{j}}{\K\varphi_{k}}_{K}=\innerproduct{\varphi_{j}}{L_{K}\varphi_{k}}_{K}=\innerproduct{\varphi_{j}}{\varphi_{k}}=\delta_{j k}$, we can show $\beta_{\theta}\in \mathcal{H}(K)$, because
\begin{align*}
    \|\beta_{\theta}\|_{K}^{2}
    &=\|M^{-\frac{1}{2}}\sum\limits_{k=M+1}^{2M}\theta_{k}\K \varphi_{k}\|_{K}^{2}  =\sum\limits_{k=M+1}^{2M}M^{-1}\theta_{k}^{2}\|\K\varphi_{k}\|_{K}^{2} \le M^{-1}\sum\limits_{k=M+1}^{2M} \|\K\varphi_{k}\|_{K}^{2} =1.
\end{align*}

\par Using the Lemma \ref{VGlemma}, there exist a set $\Theta=\{\theta^{i}\}_{i=0}^{N}\subset\{0,1\}^{M}$ such that
\begin{center}
\rm{(i)} $\theta^{0}=(0,\cdots,0)$,~~
\rm{(ii)} $H(\theta^{i}, \theta^{j}) > \frac{M}{8}$ for all $i \neq j$,~~
\rm{(iii)} $N \geq 2^{\frac{M}{8}}$.
\end{center}

For $\eta_{0}=(\bm{\alpha}_{0},\beta_{0})$, let $P^{n}_{\bm{\alpha}_{0},\beta_{0}}$ be the joint distribution on the product space of training samples $\{(Z_{i},\bm{X}_{i},Y_{i})\}_{i=1}^{n}$ generated by the true parameter $(\bm{\alpha}_{0},\beta_{0})$, where $Z_{i}=\bm{X}_{i}^{T}\bm{\alpha}_{0}+\innerproduct{Y_{i}}{\beta_{0}}+\varepsilon_{i}$, and $P_{\bm{\alpha}_{0},\beta_{0}}$ be the distribution on a single sample $(Z,\bm{X},Y)$, where $Z=\bm{X}^{T}\bm{\alpha}_{0}+\innerproduct{Y}{\beta_{0}}+\varepsilon$. \par By the independence of the training samples, for fixed $\bm{\alpha}_{0}^{*}\in\mathbb{R}^{p}$ and different $\theta,\theta^{\prime}\in\Theta$, we have
\begin{equation*}
    \log\left(\frac{\mathrm{d}P^{n}_{\bm{\alpha}_{0}^{*},\beta_{\theta^{\prime}}}}{\mathrm{d}P^{n}_{\bm{\alpha}_{0}^{*},\beta_{\theta}}}(\{(Z_{i},\bm{X}_{i},Y_{i})\}_{i=1}^n)\right)=\sum\limits_{i=1}^{n}\log\left( \frac{\mathrm{d}P_{\bm{\alpha}_{0}^{*},\beta_{\theta^{\prime}}}}{\mathrm{d}P_{\bm{\alpha}_{0}^{*},\beta_{\theta}}}(Z_{i},\bm{X}_{i},Y_{i}) \right).
\end{equation*}
Using the Assumption \ref{ass-6}, we can bound the Kullback-Leibler distance between $P^{n}_{\bm{\alpha}_{0}^{*},\beta_{\theta^{\prime}}}$ and $P^{n}_{\bm{\alpha}_{0}^{*},\beta_{\theta}}$
\begin{equation*}
    K(P^{n}_{\bm{\alpha}_{0}^{*},\beta_{\theta^{\prime}}} | P^{n}_{\bm{\alpha}_{0}^{*},\beta_{\theta}})=\sum\limits_{i=1}^{n}\operatorname{E}_{\bm{\alpha}_{0}^{*},\beta_{\theta^{\prime}}}\log\left( \frac{\mathrm{d}P_{\bm{\alpha}_{0}^{*},\beta_{\theta^{\prime}}}}{\mathrm{d}P_{\bm{\alpha}_{0}^{*},\beta_{\theta}}}\right)\le n K_{\sigma^{2}} \operatorname{E}(\innerproduct{Y}{\beta_{\theta^{\prime}}-\beta_{\theta}})^{2}.
\end{equation*}
Noticing $\innerproduct{\K\varphi_{j}}{L_{C}\K\varphi_{k}}=\innerproduct{\varphi_{j}}{T\varphi_{k}}=\tau_{k}\delta_{j k}$, we have
\begin{align*}
    \operatorname{E}(\innerproduct{Y}{\beta_{\theta^{\prime}}-\beta_{\theta}})^{2}
    &=\innerproduct{\beta_{\theta^{\prime}}-\beta_{\theta}}{L_{C}(\beta_{\theta^{\prime}}-\beta_{\theta})} =\left\langle  M^{-\frac{1}{2}}\sum\limits_{k=M+1}^{2M}(\theta^{\prime}_{k}-\theta_{k})\K \varphi_{k} , M^{-\frac{1}{2}}\sum\limits_{k=M+1}^{2M}(\theta^{\prime}_{k}-\theta_{k})L_{C}\K \varphi_{k}   \right\rangle \\
    &=M^{-1} \sum\limits_{k=M+1}^{2M} (\theta^{\prime}_{k}-\theta_{k})^{2}\tau_{k} \le M^{-1}\tau_{M}\sum\limits_{k=M+1}^{2M} (\theta^{\prime}_{k}-\theta_{k})^{2} =M^{-1}\tau_{M}H(\theta^{\prime},\theta) \le \tau_{M}= t_0M^{-2r},
\end{align*}
from which we have $ K(P^{n}_{\bm{\alpha}_{0}^{*},\beta_{\theta^{\prime}}} | P^{n}_{\bm{\alpha}_{0}^{*},\beta_{\theta}})\le t_0 n K_{\sigma^{2}} M^{-2r}$.

\par If we let $b_{0}:=(\frac{8 t_0 K_{\sigma^{2}} }{ \log2})^{\frac{1}{1+2r}} \rho^{-\frac{1}{1+2r}}$, then for any $\rho\in(0,\frac{1}{8})$, we have
\begin{equation*}
    \frac{1}{N}\sum\limits_{j=1}^{N} K(P_{\bm{\alpha}_{0}^{*},\beta_{\theta^{j}}} | P_{\bm{\alpha}_{0}^{*},\beta_{\theta^{0}}})\le t_0 n K_{\sigma^{2}} M^{-2r} \le \rho \log\left(2^{\frac{M}{8}}\right)\le \rho\log(N).
\end{equation*}

\par For $\theta\in\Theta$ and a fixed $\bm{\alpha}_{0}^{*}\in \mathbb{R}^{p}$, we consider the prediction rule $\eta_{\theta}(\bm{X},Y):=\bm{X}^{T}\bm{\alpha}_{0}^{*}+\innerproduct{Y}{\beta_{\theta}}$. For different $\theta,\theta^{\prime}\in\Theta$, when the true parameter is $(\bm{\alpha}_{0}^{*},\beta_{\theta})$, the excess prediction risk for $\eta_{\theta^{\prime}}$ is
\begin{align*}
    \mathcal{E}(\eta_{\theta^{\prime}})-\mathcal{E}(\eta_{\theta})
    &=\operatorname{E}\left[\bm{X}^{*T}(\bm{\alpha}_{0}^{*}-\bm{\alpha}_{0}^{*})+\innerproduct{Y^{*}}{\beta_{\theta^{\prime}}-\beta_{\theta}}\right]^{2}=\operatorname{E}(\innerproduct{Y^{*}}{\beta_{\theta^{\prime}}-\beta_{\theta}})^{2}\\
    &=M^{-1}\sum\limits_{k=M+1}^{2M}(\theta^{\prime}_{k}-\theta_{k})^{2}\tau_{k}\ge M^{-1}\tau_{2M} \sum\limits_{k=M+1}^{2M}(\theta^{\prime}_{k}-\theta_{k})^{2} =M^{-1}\tau_{2M}H(\theta^{\prime},\theta)\\
    &\ge M^{-1}t_0(2M)^{-2r}\frac{M}{8}=t_02^{-(2r+3)}M^{-2r}.
\end{align*}

Since $M$ is the smallest integer greater than $b_{0}n^{\frac{1}{1+2r}}$, so when $b_{0}n^{\frac{1}{1+2r}}\ge 1\Leftrightarrow n\ge \frac{\rho\log2}{8t_0K_{\sigma^{2}}}$, $M\le 2b_{0}n^{\frac{1}{1+2r}}$. Thus we obtain the lower bound for $\mathcal{E}(\eta_{\theta^{\prime}})-\mathcal{E}(\eta_{\theta})$
\begin{equation*}
    \mathcal{E}(\eta_{\theta^{\prime}})-\mathcal{E}(\eta_{\theta})\ge t_02^{-(2r+3)}(2 b_{0} n^{\frac{1}{1+2r}})^{-2r}=t_02^{-(3+4r)}(\frac{8t_0K_{\sigma^{2}}}{\log2})^{-\frac{2r}{1+2r}}\rho^{\frac{2r}{1+2r}}n^{-\frac{2r}{1+2r}}.
\end{equation*}

\par For fixed $\bm{\alpha}_{0}^{*}\in\mathbb{R}^{p}$, consider the set $\Xi:=\{(\bm{\alpha}_{0}^{*},\beta_{\theta}):\theta\in\Theta\}$.
By the Lemma \ref{KLlemma}, we have
\begin{equation*}
    \inf_{\tilde{\eta}}\sup_{\eta_{0}\in\Xi}
    P\left(\mathcal{E}(\tilde{\eta})-\mathcal{E}(\eta_{0})\ge t_02^{-4(1+r)}(\frac{8t_0K_{\sigma^{2}}}{\log2})^{-\frac{2r}{1+2r}}\rho^{\frac{2r}{1+2r}}n^{-\frac{2r}{1+2r}}\right)\ge \frac{\sqrt{N}}{1+\sqrt{N}}\left(1-2 \rho-\sqrt{\frac{2 \rho}{\log N}}\right).
\end{equation*}

Notice $\sup\limits_{\eta_{0}\in\Xi}
P\{\mathcal{E}(\tilde{\eta})-\mathcal{E}(\eta_{0})\ge\cdots\}\le
\sup\limits_{\eta_{0}\in\mathbb{R}^{p}\times\mathcal{H}(K)}
P\{\mathcal{E}(\tilde{\eta})-\mathcal{E}(\eta_{0})\ge\cdots\}$ and $\log(N)\ge \frac{\log2}{8}M$, we have the desired conclusion.
\end{proof}
\section{Conclusions and Future Studies}\label{sec-9}
\hspace{4mm}

Recently, the PFLM has raised a sizable amount of challenging problems in functional data analysis. Numerous studies focus on the asymptotic convergence rate. However, we analyze the kernel ridge estimator for the RKHS-based PFLM and obtain the non-asymptotic upper bound for the corresponding excess prediction risk. Our work to drive the optimal upper bound weakens the common assumptions in the existing literature on (partially) functional linear regressions. The optimal bound reveals that the prediction consistency holds under the setting where the number of non-functional parameters $p$ slightly increases with the sample size $n$. For fixed $p$, the convergence rate of the excess prediction risk attains the optimal minimax convergence rate under the eigenvalue decay assumption of the covariance operator.

More works could be done to study the non-asymptotic upper bound for the double penalized partially functional regressions. The penalization for the non-functional parameters could be Lasso, Elastic-net, or their generalizations. The proposed non-asymptotic upper bound is novel and substantially beneficial. It is also of interest to do non-asymptotic testing based on large deviation bounds for $\|\hatalpha{n}-\bm{\alpha}_{0}\|^{2}$ and $\|T^{\frac{1}{2}}(\hat{f}_{n}-f_{0})\|^{2}$.

\section{Acknowledgments}\label{sec-10}
\hspace{4mm}

H. Zhang is supported in part by NSFC Grant No.12101630 and the University of Macau under UM Macao Talent Programme (UMMTP-2020-01).

\bibliographystyle{apalike}
\bibliography{reference}

\begin{thebibliography}{}

\bibitem[{Abramovich} and {Grinshtein}, 2016]{abramovich2016model}
{Abramovich}, F. and {Grinshtein}, V. (2016).
\newblock Model selection and minimax estimation in generalized linear models.
\newblock {\em IEEE Transactions on Information Theory}, 62(6):3721--3730.

\bibitem[{Aneiros} et~al., 2015]{aneiros2015variable}
{Aneiros}, G., {Ferraty}, F., and {Vieu}, P. (2015).
\newblock Variable selection in partial linear regression with functional
  covariate.
\newblock {\em Statistics}, 49(6):1322--1347.

\bibitem[{Baíllo} and {Grané}, 2009]{balllo2009local}
{Baíllo}, A. and {Grané}, A. (2009).
\newblock Local linear regression for functional predictor and scalar response.
\newblock {\em Journal of Multivariate Analysis}, 100(1):102--111.

\bibitem[Blanchard and Kr{\"a}mer, 2010]{blanchard2010optimal}
Blanchard, G. and Kr{\"a}mer, N. (2010).
\newblock Optimal learning rates for kernel conjugate gradient regression.
\newblock In {\em Advances in Neural Information Processing Systems}, pages
  \,226--234.

\bibitem[{Brunel} et~al., 2016]{brunel2016non}
{Brunel}, E., {Mas}, A., and {Roche}, A. (2016).
\newblock Non-asymptotic adaptive prediction in functional linear models.
\newblock {\em Journal of Multivariate Analysis}, 143:208--232.

\bibitem[Buldygin and Kozachenko, 2000]{buldygin2000metric}
Buldygin, V.~V. and Kozachenko, I.~V. (2000).
\newblock {\em Metric characterization of random variables and random
  processes}, volume 188.
\newblock American Mathematical Soc.

\bibitem[Cai et~al., 2006]{cai2006prediction}
Cai, T.~T., Hall, P., et~al. (2006).
\newblock Prediction in functional linear regression.
\newblock {\em The Annals of Statistics}, 34(5):\,2159--2179.

\bibitem[Cai and Yuan, 2012]{cai2012minimax}
Cai, T.~T. and Yuan, M. (2012).
\newblock Minimax and adaptive prediction for functional linear regression.
\newblock {\em Journal of the American Statistical Association},
  107(499):\,1201--1216.

\bibitem[Caponnetto and De~Vito, 2007]{caponnetto2007optimal}
Caponnetto, A. and De~Vito, E. (2007).
\newblock Optimal rates for the regularized least-squares algorithm.
\newblock {\em Foundations of Computational Mathematics}, 7(3):331--368.

\bibitem[Cardot et~al., 2003]{cardot2003spline}
Cardot, H., Ferraty, F., and Sarda, P. (2003).
\newblock Spline estimators for the functional linear model.
\newblock {\em Statistica Sinica}, pages \,571--591.

\bibitem[{Cucker} and {Smale}, 2001]{cucker2001on}
{Cucker}, F. and {Smale}, S. (2001).
\newblock On the mathematical foundations of learning.
\newblock {\em Bulletin of the American Mathematical Society}, 39(1):1--49.

\bibitem[Cui et~al., 2020]{cui2020partially}
Cui, X., Lin, H., and Lian, H. (2020).
\newblock Partially functional linear regression in reproducing kernel hilbert
  spaces.
\newblock {\em Computational Statistics \& Data Analysis}, page 106978.

\bibitem[{Du} and {Wang}, 2014]{du2014penalized}
{Du}, P. and {Wang}, X. (2014).
\newblock Penalized likelihood functional regression.
\newblock {\em Statistica Sinica}, 24(2):1017--1041.

\bibitem[{Grenander}, 1950]{grenander1950stochastic}
{Grenander}, U. (1950).
\newblock Stochastic processes and statistical inference.
\newblock {\em Arkiv för Matematik}, 1(3):\,195--277.

\bibitem[Hall et~al., 2007]{hall2007methodology}
Hall, P., Horowitz, J.~L., et~al. (2007).
\newblock Methodology and convergence rates for functional linear regression.
\newblock {\em The Annals of Statistics}, 35(1):\,70--91.

\bibitem[Horn and Johnson, 2012]{horn2012matrix}
Horn, R.~A. and Johnson, C.~R. (2012).
\newblock {\em Matrix analysis}.
\newblock Cambridge university press.

\bibitem[Hsing and Eubank, 2015]{hsing2015theoretical}
Hsing, T. and Eubank, R. (2015).
\newblock {\em Theoretical foundations of functional data analysis, with an
  introduction to linear operators}, volume 997.
\newblock John Wiley \& Sons.

\bibitem[Jolliffe, 1982]{jolliffe1982note}
Jolliffe, I.~T. (1982).
\newblock A note on the use of principal components in regression.
\newblock {\em Journal of the Royal Statistical Society: Series C (Applied
  Statistics)}, 31(3):\,300--303.

\bibitem[Kokoszka and Reimherr, 2017]{kokoszka2017introduction}
Kokoszka, P. and Reimherr, M. (2017).
\newblock {\em Introduction to functional data analysis}.
\newblock CRC Press.

\bibitem[{Kong} et~al., 2016]{kong2016partially}
{Kong}, D., {Xue}, K., {Yao}, F., and {Zhang}, H.~H. (2016).
\newblock Partially functional linear regression in high dimensions.
\newblock {\em Biometrika}, 103(1):147--159.

\bibitem[Liu and Li, 2020]{liu2020non}
Liu, Z. and Li, M. (2020).
\newblock Non-asymptotic analysis in kernel ridge regression.
\newblock {\em arXiv preprint arXiv:2006.01350}.

\bibitem[Lv and Feng, 2012]{lv2012integral}
Lv, S.-G. and Feng, Y.-L. (2012).
\newblock Integral operator approach to learning theory with unbounded
  sampling.
\newblock {\em Complex Analysis and Operator Theory}, 6(3):533--548.

\bibitem[Maurer, 2012]{maurer2012thermodynamics}
Maurer, A. (2012).
\newblock Thermodynamics and concentration.
\newblock {\em Bernoulli}, 18(2):434--454.

\bibitem[Maurer and Pontil, 2021]{maurer2021concentration}
Maurer, A. and Pontil, M. (2021).
\newblock Concentration inequalities under sub-gaussian and sub-exponential
  conditions.
\newblock {\em Advances in Neural Information Processing Systems},
  34:7588--7597.

\bibitem[Okamoto, 1973]{okamoto1973distinctness}
Okamoto, M. (1973).
\newblock Distinctness of the eigenvalues of a quadratic form in a multivariate
  sample.
\newblock {\em The Annals of Statistics}, pages 763--765.

\bibitem[Ostrovskii and Bach, 2021]{ostrovskii2018finite}
Ostrovskii, D.~M. and Bach, F. (2021).
\newblock Finite-sample analysis of $ m $-estimators using self-concordance.
\newblock {\em Electronic Journal of Statistics}, 15(1):326--391.

\bibitem[Preda, 2007]{preda2007regression}
Preda, C. (2007).
\newblock Regression models for functional data by reproducing kernel hilbert
  spaces methods.
\newblock {\em Journal of statistical planning and inference},
  137(3):\,829--840.

\bibitem[Ramsay, 1982]{ramsay1982data}
Ramsay, J. (1982).
\newblock When the data are functions.
\newblock {\em Psychometrika}, 47(4):\,379--396.

\bibitem[Ramsay and Silverman, 2007]{ramsay2007applied}
Ramsay, J.~O. and Silverman, B.~W. (2007).
\newblock {\em Applied functional data analysis: methods and case studies}.
\newblock Springer.

\bibitem[{Reimherr} et~al., 2018]{reimherr2018optimal}
{Reimherr}, M.~L., {Sriperumbudur}, B.~K., and {Taoufik}, B. (2018).
\newblock Optimal prediction for additive function-on-function regression.
\newblock {\em Electronic Journal of Statistics}, 12(2):4571--4601.

\bibitem[Schechter, 2001]{schechter2001principles}
Schechter, M. (2001).
\newblock {\em Principles of functional analysis}.
\newblock Number~36. American Mathematical Soc.

\bibitem[Shin, 2009]{shin2009partial}
Shin, H. (2009).
\newblock Partial functional linear regression.
\newblock {\em Journal of Statistical Planning and Inference},
  139(10):\,3405--3418.

\bibitem[Sun, 2005]{sun2005mercer}
Sun, H. (2005).
\newblock Mercer theorem for rkhs on noncompact sets.
\newblock {\em Journal of Complexity}, 21(3):\,337--349.

\bibitem[Sun et~al., 2017]{Sun2017GraphicalNO}
Sun, Q., Tan, K.~M., Liu, H., and Zhang, T. (2017).
\newblock Graphical nonconvex optimization for optimal estimation in gaussian
  graphical models.
\newblock In {\em ICML 2018}.

\bibitem[Tong and Ng, 2018]{tong2018analysis}
Tong, H. and Ng, M. (2018).
\newblock Analysis of regularized least squares for functional linear
  regression model.
\newblock {\em Journal of Complexity}, 49:\,85--94.

\bibitem[Tsybakov, 2008]{tsybakov2008}
Tsybakov, A.~B. (2008).
\newblock {\em Introduction to nonparametric estimation}.
\newblock Springer Science \& Business Media.

\bibitem[Wahba, 1990]{wahba1990spline}
Wahba, G. (1990).
\newblock {\em Spline models for observational data}.
\newblock SIAM.

\bibitem[{Wahl}, 2018]{wahl2018a}
{Wahl}, M. (2018).
\newblock A note on the prediction error of principal component regression.
\newblock {\em arXiv preprint arXiv:1811.02998}.

\bibitem[Wang et~al., 2016]{wang2016functional}
Wang, J.-L., Chiou, J.-M., and M{\"u}ller, H.-G. (2016).
\newblock Functional data analysis.
\newblock {\em Annual Review of Statistics and Its Application}, 3:\,257--295.

\bibitem[Yang et~al., 2020]{yang2020non}
Yang, Y., Shang, Z., and Cheng, G. (2020).
\newblock Non-asymptotic analysis for nonparametric testing.
\newblock In {\em Conference on Learning Theory}, pages \,3709--3755.

\bibitem[Yao et~al., 2005]{yao2005functional}
Yao, F., M{\"u}ller, H.-G., and Wang, J.-L. (2005).
\newblock Functional linear regression analysis for longitudinal data.
\newblock {\em The Annals of Statistics}, pages \,2873--2903.

\bibitem[Yurinsky, 2006]{yurinsky2006sums}
Yurinsky, V. (2006).
\newblock {\em Sums and Gaussian vectors}.
\newblock Springer.

\bibitem[Zhang and Chen, 2021]{zhang2020concentration}
Zhang, H. and Chen, S.~X. (2021).
\newblock Concentration inequalities for statistical inference.
\newblock {\em Communications in Mathematical Research}, 37(1):1--85.

\bibitem[Zhang, 2005]{zhang2005learning}
Zhang, T. (2005).
\newblock Learning bounds for kernel regression using effective data
  dimensionality.
\newblock {\em Neural Computation}, 17(9):\,2077--2098.

\bibitem[Zhou et~al., 2022]{zhou2022functional}
Zhou, H., Yao, F., and Zhang, H. (2022).
\newblock Functional linear regression for discretely observed data: From ideal
  to reality.
\newblock {\em Biometrika}.

\bibitem[Zhu et~al., 2019]{zhu2019estimation}
Zhu, H., Zhang, R., Yu, Z., Lian, H., and Liu, Y. (2019).
\newblock Estimation and testing for partially functional linear
  errors-in-variables models.
\newblock {\em Journal of Multivariate Analysis}, 170:\,296--314.

\bibitem[Zhuang and Lederer, 2018]{zhuang2018maximum}
Zhuang, R. and Lederer, J. (2018).
\newblock Maximum regularized likelihood estimators: A general prediction
  theory and applications.
\newblock {\em Stat}, 7(1):e186.

\end{thebibliography}

\end{document}